\documentclass{article}
\usepackage[utf8]{inputenc}

\usepackage{geometry}
\usepackage{amsmath}
\usepackage{amsfonts}
\usepackage{mathrsfs}
\usepackage{amsthm}
\usepackage{amssymb}
\usepackage{hyperref}
\usepackage[capitalize]{cleveref}
\usepackage{extarrows}
\usepackage{enumitem}
\usepackage{bbm}
\usepackage{upgreek}
\usepackage{mathtools}
\usepackage{wasysym}
\usepackage{xfrac}
\usepackage{graphicx}
\usepackage{float}
\usepackage{multirow}
\usepackage{caption}
\usepackage{subcaption}

\newtheorem{theorem}{Theorem}[section]
\newtheorem{corollary}[theorem]{Corollary}
\newtheorem{lemma}[theorem]{Lemma}

\theoremstyle{definition}
\newtheorem{definition}{Definition}[section]
\theoremstyle{remark}
\newtheorem{remark}{Remark}[section]

\newtheorem{claim}{Claim}[theorem]

\newcommand{\N}{\mathbb{N}}
\newcommand{\R}{\mathbb{R}}

\usepackage[nottoc]{tocbibind}

\usepackage{authblk}

\renewcommand{\d}[1]{\ensuremath{\operatorname{d}\!{#1}}}

\title{A Modified Trapezoidal Rule for a Class of Weakly Singular Integrals in $n$ Dimensions}
\author[1]{Senbao Jiang\footnote{Corresponding author: sjiang23@hawk.iit.edu}}
\author[1]{Xiaofan Li\footnote{lix@iit.edu}}
\affil[1]{Department of Applied Mathematics, Illinois Institute of Technology}
\date{\today}

\begin{document}

\maketitle
\begin{abstract}
    In this paper we propose and analyze a general arbitrarily high-order modified trapezoidal rule for a class of weakly singular integrals of the forms $I = \int_{\R^n}\phi(x)s(x)\d x$ in $n$ dimensions, where $\phi\in C_c^N(\R^n)$ for some sufficiently large $N$ and $s$ is the weakly singular kernel. The admissible class of weakly singular kernel requires $s$ satisfies dilation and symmetry properties and is large enough to contain functions of the form $\frac{P(x)}{|x|^r}$ where $r > 0$ and $P(x)$ is any monomials such that $\deg P < r < \deg P + n$. The modified trapezoidal rule is the singularity-punctured trapezoidal rule added by correction terms involving the correction weights for grid points around singularity. Correction weights are determined by enforcing the quadrature rule exactly evaluates some monomials and solving corresponding linear systems. A long-standing difficulty of these type of methods is establishing the non-singularity of the linear system, despite strong numerical evidences. By using an algebraic-combinatorial argument, we show the non-singularity always holds and prove the general order of convergence of the modified quadrature rule. We present numerical experiments to validate the order of convergence.
\end{abstract}

\section{Introduction}
Numerical integration is a classic topic in numerical analysis and it is still an area of active research. Classical methods such as trapezoidal rule, Simpson rule and Gaussian quadrature have become integral part of standard textbooks, e.g. \cite{kincaid2009numerical}. Many of these quadrature require some regularity of the integrand and therefore can be generalized to higher dimensions by iterated integral justified by Fubini's theorem. It is not always the case for weakly singular integrals. There are many numerical methods for weakly singular integrals in low dimensions, i.e. one or two or three dimension. %For example, the singularity subtraction or removal methods, see \cite{10.2307/2157208,Bruno2001AFH,Atkinson2004QUADRATUREOS,1514571,Sidi2005ApplicationOC,Mousavi2009GeneralizedDT} and 
Among them, a class of methods based on modifying the trapezoidal rules are popular, see \cite{ROKHLIN199051,Alpert1990RapidlyConvergentQF,doi:10.1137/S0036142995287847,AGUILAR20021031, Keast1979OnTS, doi:10.1137/S0036142995287847,MarinTornberg2014, JIANG2022127236}. 

Rokhlin \cite{ROKHLIN199051} was first to propose singularity-corrected trapezoidal rule and Alpert \cite{Alpert1990RapidlyConvergentQF} and Kapur and Rokhlin \cite{doi:10.1137/S0036142995287847} further improved the method. Their quadrature rules are designed for functions of the form $f(x) = \phi(x) s(x) + \Psi(x)$ or $f(x) = \phi(x) s(x)$ in 1D, where $\phi(x),\Psi(x)$ are regular functions and $s(x)$ is singular function with isolated singularity such as $s(x) = |x|^{\gamma},\gamma>-1$ or $s(x) = \log(x)$. Aguilar and Chen designed singularity-boundary-corrected trapezoidal rule for singular kernel $s(x) = \log(x)$ in 2D \cite{AGUILAR20021031} and $s(x) = \sfrac{1}{|x|}$ in 3D \cite{AGUILAR2005625}, in particular, their singularity correction is introducing correction terms in the vicinity of the singularity, involving correction weights and values of the regular part of the integrand. Correction weights are computed through enforcing the singularity-corrected trapezoidal rule exactly evaluate monomials up to certain degree. They observed high-order convergence of the rule but do not offer any proofs. Marin et al \cite{MarinTornberg2014} proposed an increasingly high-order modified trapezoidal rule to weakly singular integrals $\int\phi(x)s(x)\d x$ with sufficiently smooth part $\phi(x)$ with compact support and singular parts $s(x) = |x|^{\gamma},\gamma>-1$ in 1D and $s(x) = \sfrac{1}{|x|}$ in 2D and conducted rigorous analysis of order of convergence. Recently, Jiang and Li \cite{JIANG2022127236} further developed this method to weakly singular integrals $\int_{\R^2} \phi(x)s_{ij}(x)\d x$ with $\phi(x)\in C_c^{2p+4}(\R^2)$, $s(x) = \frac{x_ix_j}{|x|^{2+\alpha}}$ where $1\leq i,j\leq 2$ and $\alpha\in (0,2)$. They proved the order of convergence is $2p+4-\alpha$ where $p\geq 1$ is associated with total number of correction weights. A further application to numerical fractional laplacian in 2D can be found at \cite{jiang2022solving}. This type of methods have a common difficulty: proving the linear systems of correction weights have a unique solution. We refer this difficulty as \emph{non-singularity problem}. One of the main convergence theorems in \cite{MarinTornberg2014} is conditional upon the non-singularity. Jiang and Li \cite{JIANG2022127236} proposed an algebraic-combinatorial method to overcome their non-singularity problem in 2D, regardless of the number of the correction weights. Their method can be easily adapted to non-singularity problem of Marin et al \cite{MarinTornberg2014} and hence the convergence theorem in \cite{MarinTornberg2014} becomes unconditional.

In this paper we generalize the modified trapezoidal rule in \cite{JIANG2022127236} to a class of weakly singular integrals in arbitrary $n$ dimensions and provide corresponding convergence analysis. The quadrature rules apply to any weakly singular integrals with any singular part $s$ satisfying dilation property \cref{eq:dilation cond} and symmetry property \cref{eq:symmetry}, described in \cref{subsection:Mod Trapz}. The class of admissible weakly singular kernels is large. We prove the quadrature rules attain high order of convergence, provided that the regular part of the integrands satisfy some smoothness criteria. In doing so, we completely resolve the non-singularity problem in arbitrary $n$ dimensions by a generalized algebraic-combinatorial argument used in \cite{JIANG2022127236}, thereby proving the modified trapezoidal rule is universally feasible.

We organize the paper as follows. In \cref{section:Mod Trapez}, we introduce the admissible class of the weakly singular integrals and the modified trapezoidal rules in $n$ dimensions, together with their associated matrix formulations for determining correction weights. We prove the main convergence theorems in \cref{section:Convergence analysis}. We address the non-singularity problem arising in matrix formulation of \cref{section:Non-singularity}. \cref{section:Num Res} presents numerical results for order of convergences and associated correction weights. The final \cref{sec:conclusion} summarizes the paper with possible future directions.

\section{General Modified Trapezoidal Rules} \label{section:Mod Trapez}

\subsection{Notations}
Through-out this paper, the natural number $n$ is the dimension, $h$ is the mesh size, lowercase letters such as $x,y,z$ are scalars or vectors and bold uppercase letters such as $\boldsymbol{A},\boldsymbol{B}$ are matrices. The Euclidean norms $|x|\coloneqq \sqrt{\sum_{i = 1}^n x_i^2}$ and $|x|_1:=\sum_{i=1}^n|x_i|$. The natural number with zero $\mathbb{N}_0\coloneqq \{0,1,2,\cdots\}$.
Multi-index notations appeared frequently in this paper. In particular, we use $p\leq q$ for $p = (p_1,\cdots,p_n),\ q = (q_1,\cdots,q_n)$ if and only if $p_i\leq q_i$ for all $1\leq i\leq n$ and $x^{\gamma} \coloneqq \prod_j x_j^{\gamma_j}$ for any multi-index $\gamma = (\gamma_1,\cdots,\gamma_n)$. $C_c^k(\R^n)$ is the space of $k$-order smooth functions with compact support, $\mathcal{S}(\R^n)$ is the space of Schwartz functions and $S^{n - 1}\coloneqq\{x\in\R^n:|x|_2 = 1\}$ is the $n-1$ dimensional unit sphere.

\subsection{Modified Trapezoidal Rule}\label{subsection:Mod Trapz}
By a weakly singular kernel $s$ on $\R^n\setminus\{0\}$ we mean there exists $L\in\R$ such that\begin{align*}
    \lim_{h\to0 }\int_{h<|x|\leq 1}s(x)\d x = L.
\end{align*} In this paper we assume the weakly singular kernel $s$ is smooth on $\R^n\setminus\{0\}$ and satisfies the following properties: \begin{enumerate}
    % \item \textbf{The Gradient Condition:} there exists a $\delta\in(0,n)$ such that for any multi-index $k\in \N_0^n$, there exists $A_k>0$ such that \begin{align}
    %     |\partial^k s(x)|\leq \frac{A_k}{|x|^{n-\delta+|k|_1}},\quad \forall x\in\R^n\setminus\{0\}. \label{eq:grad cond}
    % \end{align} 
    \item \textbf{The Dilation Property:} There exists a $\delta\in(0,n)$ and a non-zero smooth function $\rho:\R^n\rightarrow\R$ such that \begin{align}
        s(x) = |x|^{\delta-n}\rho\left(\frac{x}{|x|}\right),\quad \forall x\in \R^n\setminus\{0\}. \label{eq:dilation cond}
    \end{align} It follows immediately from \cref{eq:dilation cond} that $s(hx) = h^{\delta - n}s(x)$ for all $x\in\R^n\setminus\{0\}$ and $h > 0$.
    \item \textbf{The Symmetry Property:} There exists an integer $0\leq \kappa\leq n$ such that for all $x\in\R^n\setminus\{0\}$ \begin{align}
        s(x_1,\cdots,x_j,\cdots, x_n) = \begin{cases} \label{eq:symmetry}
        -s(x_1,\cdots,-x_{j},\cdots,x_n), & 1\leq j\leq \kappa \\
        s(x_1,\cdots,-x_{j},\cdots,x_n), & \kappa<j\leq n
        \end{cases}.
    \end{align}
    % for each $1\leq j\leq \kappa$ \begin{align}
    %     s(x_1,\cdots,-x_{j},\cdots,x_n) = -s(x_1,\cdots,x_n),\quad \forall x\in\R^n\setminus\{0\}, \label{eq:odd variable}
    % \end{align} and for each $\kappa<j\leq n$ \begin{align}
    %     s(x_1,\cdots,-x_{j},\cdots,x_n) = s(x_1,\cdots,x_n),\quad \forall x\in\R^n\setminus\{0\}. \label{eq:even variable}
    % \end{align}
\end{enumerate}
\begin{remark}
One can easily show that weakly singular kernels of the form $s(x) = \frac{x_1^{2m+1} x_j^{2k}}{|x|^r}$ with $m, k\in \N_0$, $1<j\leq n$ and $2(m + k) + 1 < r < 2(m + k) + 1 + n$ satisfy the properties in \cref{eq:dilation cond,,eq:symmetry}. With proper reassignment of variables, one can similarly show that any weakly singular kernels of the form $\frac{P(x)}{|x|^r}$ in which $P(x)$ is a monomial such that $\deg P < r < \deg P + n$ satisfy the properties in \cref{eq:dilation cond,,eq:symmetry}.
\end{remark}
We focus on weakly singular integrals of the form \begin{align}
    I\coloneqq \int_{\R^n}\phi(x)s(x)\d x, \label{eq:weakly singular integral}
\end{align} where $\phi\in C_c^N(\mathbb{R}^n)$ for some $N$ to be determined and $s$ be the weakly singular kernel satisfying \cref{eq:dilation cond,,eq:symmetry}. For any compactly supported function $f$ on $\mathbb{R}^n$, the punctured-hole trapezoidal rule is defined by \begin{align}\label{Punctured_Hole_Trapez}
    T_h^0[f]\coloneqq h^n\sum_{\beta\in\mathbb{Z}^n\setminus\{0\}} f(\beta h).
\end{align} Let $p\in\mathbb{N}_0$, we introduce modified trapezoidal rule in $\mathbb{R}^n$ \begin{align}\label{FirstCorrectedTrapezRule}
Q_h^p[\phi\cdot s]\coloneqq T_h^0[\phi\cdot s]+h^\delta\sum_{\beta\in\mathcal{M}_{n,p}}\bar{\omega}_{\beta}\phi(\beta h),
\end{align} where $\bar{\omega}_{\beta}$'s are the correction weights to be defined and \begin{align}
    \mathcal{M}_{n,p} = \left\{\beta\in\mathbb{Z}^n:|\beta|_1\leq p, \prod_{1\leq j\leq \kappa}\beta_j\not= 0\right\}.
\end{align} Here, $\kappa$ is the number of arguments in the weakly singular kernel that have odd symmetry as defined in \cref{eq:symmetry}, and $h^\delta$ is the leading order error of the punctured-hole trapezoidal rule for approximating the weakly singular integral \cref{eq:weakly singular integral}. To see this, we can roughly evaluate \begin{align*}
    \int_{(-h,h)^n}|\phi(x)|s(x)\ \d x &\sim \int_{B(0,h)}|\phi(x)|s(x)\ \d x\lesssim \int_{B(0,h)}\frac{\d x}{|x|^{n-\delta}}\ \lesssim \int_0^h r^{-1+\delta}\ \d r \sim h^{\delta},
\end{align*} where $\sim$ means on the same order as $h\to 0$.

We now elaborate on the method for determining the correction weights $\bar{\omega}_{\beta}$. To facilitate the description, we denote some sets on the grid $\mathcal{M}(n,p)$ by \begin{align}
    \mathcal{I}(n,p) &\coloneqq \{\beta\in\mathcal{M}_{n,p}:\beta_j \geq 0\ \forall j\}, \label{eq:grids}\\
    \mathcal{G}_\beta & \coloneqq \big\{((-1)^{k_1} \beta_1,(-1)^{k_2}\beta_2,\dots,(-1)^{k_n}\beta_n):k_j\in \{0,1\}, j = 1,\dots,n\big\},\ \forall \beta\in\mathcal{I}(n,p), \label{eq:G_eta}\\
     \mathcal{G} & \coloneqq \{\mathcal{G}_\beta: \beta\in\mathcal{I}(n,p)\}.
    \end{align} It is clear that $|\mathcal{G}| =|\mathcal{I}(n,p)|$. We write $I_{n,p}\coloneqq |\mathcal{I}(n,p)|$. By symmetry of $s$, we impose the same symmetry on the weights \begin{align}
    \bar{\omega}_{\beta_1,\dots,\beta_n} = \text{sgn}\left(\prod_{1\leq j\leq \kappa}\beta_j\right)\bar{\omega}_{|\beta_1|,\dots, |\beta_n|} , \quad \forall\beta = (\beta_1,\dots,\beta_n)\in \mathcal{M}_{n,p}.
\end{align} Figure \ref{fig:2D grid} presents the location of the correction weights $\bar{\omega}_{\beta}$ in $\mathcal{I}(2,8)$ when $\kappa = 0$ and $\kappa = 1$. \begin{figure}[h!]
    \centering
    \subfloat[$\kappa = 0$]{\includegraphics[width = 0.45\textwidth]{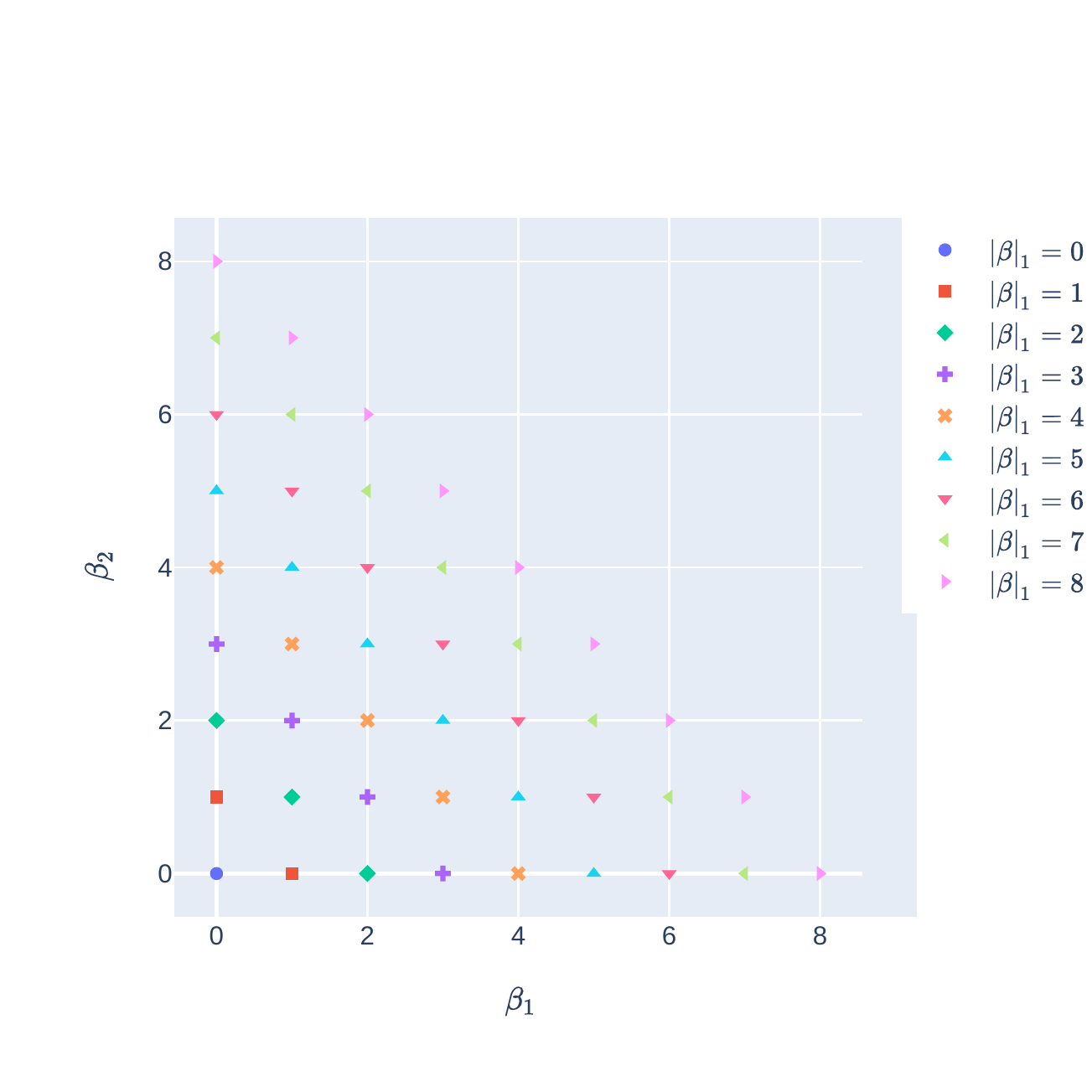}}
    \subfloat[$\kappa = 1$]{\includegraphics[width = 0.45\textwidth]{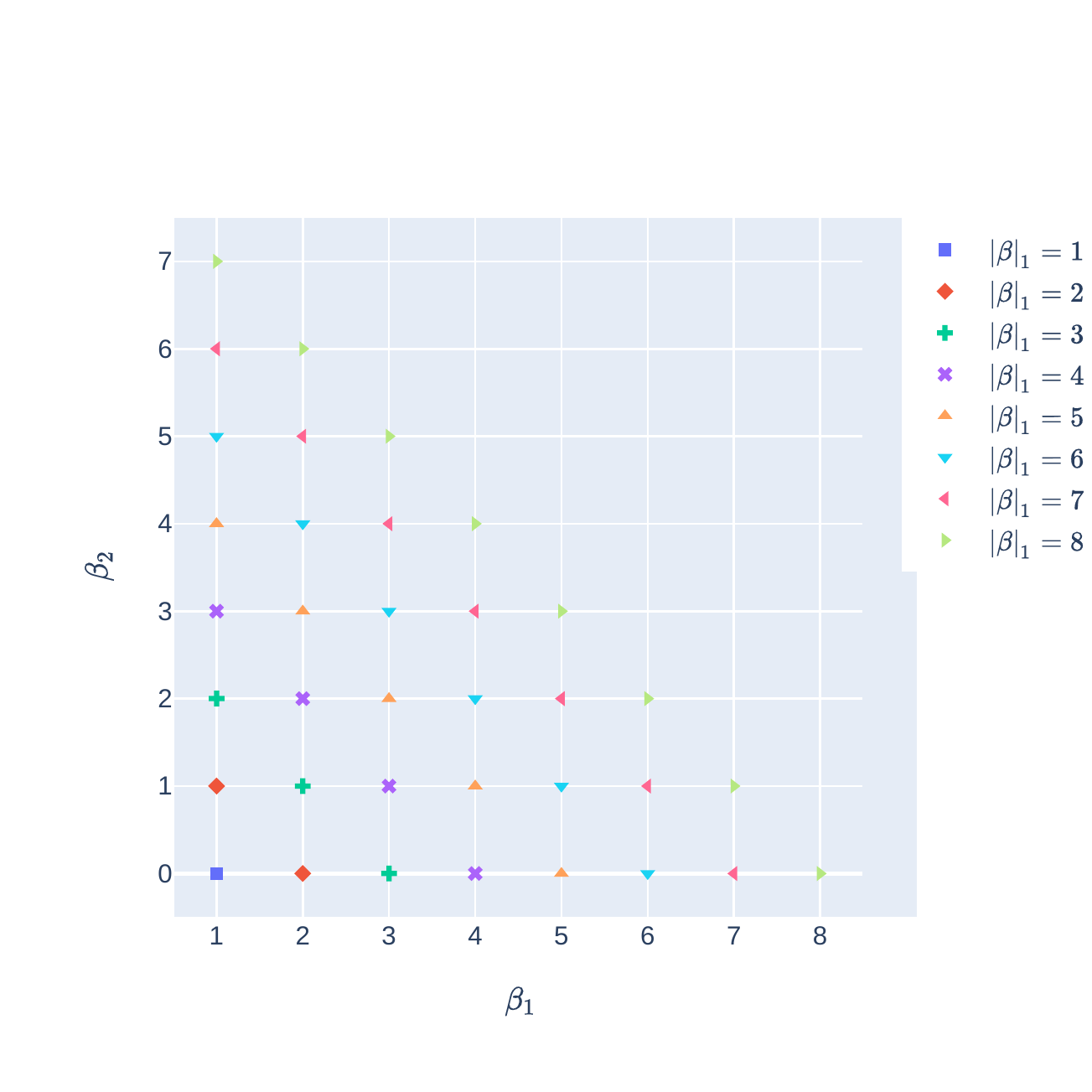}}
    \caption{The location of the correction weights in $\mathcal{I}(2,8)$ with (a) $\kappa = 0$ and (b) $\kappa = 1$.}
    \label{fig:2D grid}
\end{figure}

For each $h>0$ we require the modified trapezoidal rule $Q_h^p$ in \cref{FirstCorrectedTrapezRule} with weights $\omega_{\beta}(h)$ evaluate the following integrals exactly: for each $\xi\in \mathcal{I}(n,p)$ \begin{align}\label{FirstCoeffLinearSystem}
    \int_{\mathbb{R}^n} g(x)s(x)x^{2\xi-\sum_{j=1}^{\kappa} e_j}\, \d x &= T_h^0[g\cdot s\cdot x^{2\xi-\sum_{j=1}^{\kappa} e_j}]\nonumber \\
    &+h^\delta\sum_{\eta\in\mathcal{I}(n,p)}\omega_{\eta}(h)g(\eta h)\sum_{\beta\in\mathcal{G}_\eta}\text{sgn}\left(\prod_{j=1}^{\kappa}\beta_j\right)(\beta h)^{2\xi-\sum_{j=1}^{\kappa}e_j}, 
\end{align} where $g$ is a radially symmetric, smooth function with compact support such that $g \not \equiv 0$. The weights $\bar{\omega}_\beta$ are the limits of the $\omega_{\beta}(h)$ as $h\to 0$.

It is clear that the system of equations \cref{FirstCoeffLinearSystem} is $I_{n,p}\times I_{n,p}$ linear system for the weights $\omega_{\eta}$ and we re-write the modified trapezoidal rule (\ref{FirstCorrectedTrapezRule}) into \begin{align} \label{SecondCorrectedTrapezRule}
    Q_h^p[\phi\cdot s] = T_h^0[\phi\cdot s] + A_h^p[\phi]
\end{align}where \begin{align} \label{A^p}
    A^p_h[\phi] = h^\delta\sum_{\eta\in\mathcal{I}(n,p)}\bar{\omega}_\eta\sum_{\beta\in\mathcal{G}_\eta}\text{sgn}\left(\prod_{j=1}^{\kappa}\beta_j\right)\phi(\beta h). 
\end{align} 
To formulate \cref{FirstCoeffLinearSystem} in matrix form, we re-index the set $\mathcal{I}(n,p)=(\eta_i)_{1\leq i\leq I_{n,p}}$. There are more than one way to index the set $\mathcal{I}(n,p)$ and the indexing plays an important role in proving the linear system has a unique solution. We specify the indexing later. By the one-to-one correspondence between $\mathcal{G}$ and $\mathcal{I}(n,p)$, we index $\mathcal{G}$  by $(\mathcal{G}_j\coloneqq \mathcal{G}_{\eta_j})_{1\leq j\leq I_{n,p}}$, for each $\eta_j\in\mathcal{I}(n,p), 1\leq j\leq I_{n,p}$. We define $\boldsymbol{K},\boldsymbol{G}\in\mathbb{R}^{I_{n,p}\times I_{n,p}}$ such that for all $1\leq i,j\leq I_{n,p}$ \begin{align}\label{MatrixK}
    K_{i,j} = \sum_{\beta\in\mathcal{G}_j}\text{sgn}\left(\prod_{j=1}^{\kappa}\beta_j\right)\beta^{2\xi_i-\sum_{j=1}^{\kappa} e_j},\ G_{i,j} = g(\beta h)\delta_{i,j},\ \beta\in\mathcal{G}_j.
\end{align} We call $\boldsymbol{K}$ the coefficient matrix. Now, let \begin{align}
     \omega(h) = (\omega_1(h),\dots,\omega_{I_{n,p}}(h))^T,
\end{align} the linear system \cref{FirstCoeffLinearSystem} becomes \begin{align}\label{MatrixLinearSystem}
    \boldsymbol{K}\boldsymbol{G}(h)\omega(h) = c(h),
\end{align} where the right-hand side of \cref{MatrixLinearSystem} is given by \begin{align}
    c_i(h) &= \frac{1}{h^{2|\xi_i|_1-\kappa+\delta}}\left(\int g(x)s(x)x^{2\xi_i-\sum_{j=1}^{\kappa}e_j}\d x - T_h^0[g\cdot s\cdot x^{2\xi_i-\sum_{j=1}^{\kappa}e_j}]\right),\quad 1\leq i\leq I_{n,p}, \label{RHSofLinearSystem} \\
    c(h) &= (c_1(h),\dots,c_{I_{n,p}}(h))^T.
\end{align} We solve the linear system \cref{MatrixLinearSystem} for each $h>0$ and the coefficients $\bar{\omega}_\eta:\eta\in\mathcal{I}(n,p)$ are the limits of the solution of the linear system \cref{MatrixLinearSystem}: $\bar{\omega}=\lim_{h\to 0}\omega(h)$, provided the following claims hold \begin{itemize}
    \item the limit of right-hand side of equation \cref{RHSofLinearSystem} exists,
    \item \emph{Non-singularity problem}: $\boldsymbol{K}$ is non-singular,
    \item the limit of $\boldsymbol{G}(h)$ exists, denoted by $\boldsymbol{G}$ and $\boldsymbol{G}$ are non-singular.
\end{itemize} The first and last claims are proved in \cref{section:Convergence analysis} while the non-singularity problem for $\boldsymbol{K}$ is addressed in \cref{section:Non-singularity}.

%%%%%%%%%%%%%%%%%%%%%%%%%%%%%%%%%%%%%%%%%%%%%%%%%%%%%%%%%%%%%%%%%%%%%%%%%%%%%%%%%%%%%%%%%%%%%%%%%%%%%%%%%%%%%%%%%%%%%%%%%%%%%%%%%%%%%%%%%%%%%%%%%%%%%%%%%%%%%%%%
\section{Analysis of Orders of Accuracy} \label{section:Convergence analysis}
The next theorem gives the order of accuracy of the modified trapezoidal rule \cref{SecondCorrectedTrapezRule} and the convergence rate of the weights $\omega(h)$ as $h\to 0$. In this section, $s$ is a weakly singular integral satisfying the dilation and symmetry properties, where constants $\delta\in(0,n)$ and $\kappa\in\{0,\dots,n\}$ are defined at \cref{eq:dilation cond,,eq:symmetry}, respectively.
\begin{theorem}\label{theorem:MainTheorem}
Let $n\in\N$ and $p\in\mathbb{N}_0$ satisfy $2p\geq \kappa$. Given $\phi\in C_c^{N}(\mathbb{R}^n)$ with $N>\max(2p-\kappa+2+\delta,n)$ or $\phi\in \mathcal{S}(\R^n)$. Assume $g\in \mathcal{S}(\R^n)$ such that $g$ is radially symmetric, $g(0)=1$ and $\partial^{k} g(0)=0$ for all multi-indices $|k|_1 =1,\dots, 2p-\kappa+1$. Let $\omega(h)$ be the solution of \cref{MatrixLinearSystem} for this $g$, then $\omega(h)$ converges to some $\bar{\omega}\in\mathbb{R}^{I_{n,p}}$ such that \begin{align}  
    |\omega(h)-\bar{\omega}|_2= O(h^{2p-\kappa+2}).
\end{align} Moreover, there exists $C>0$ such that \begin{align} \label{AccuracyofCorrectedTrapezRule}
    \left|Q_h^p[\phi\cdot s] - \int_{\mathbb{R}^n\setminus\{0\}}\phi(x)\cdot s(x)\d x\right|\leq C h^{2p-\kappa+2+\delta}.
\end{align}
\end{theorem} 
We will need some preliminary results before we prove \cref{theorem:MainTheorem}.
\begin{lemma}[Poisson Summation Formula]\label{lemma:PoissonSummationFormula}
Let $f$ be a continuous function on $\mathbb{R}^n$ which satisfies \begin{align}
    |f(x)|\leq \frac{C}{(1+|x|)^{n+\nu}} \quad \text{for some}\ C,\nu>0 \ \text{and for all}\ x\in\mathbb{R}^n,
\end{align} and whose Fourier transform $\widehat{f}$ restricted on $\mathbb{Z}^n$ satisfies \begin{align} \label{CondPoissonSummation}
    \sum_{m\in\mathbb{Z}^n}|\widehat{f}(m)|<\infty,
\end{align} then \begin{align} \label{PoissonSummationFormula3}
    \sum_{m\in\mathbb{Z}^n}\widehat{f}(m) = \sum_{k\in\mathbb{Z}^n}f(k).
\end{align}
\end{lemma} For a proof of \cref{lemma:PoissonSummationFormula}, see \cite[Theorem 3.2.8]{grafakos2014classical}.
\begin{remark}
It is not hard to verify that if either $f\in C_c^{n+1}(\mathbb{R}^n)$ or $f\in\mathcal{S}(\R^n)$, then $f$ satisfies the hypotheses of Poisson Summation formula.
\end{remark}

Let $\Psi\in C^\infty(\mathbb{R}^n)$ be a smooth, radially symmetric cut-off function such that \begin{align}\label{Cut-offFunction}
    \Psi(x) = \left\{ \begin{array}{cc}
        0, & |x|\leq \frac{1}{2} \\
        1, & |x|\geq 1
    \end{array}\right. .
\end{align} There is a elegant way to construct such $\Psi$. Let $\psi\in C_c^{\infty}(\mathbb{R})$ such that $supp(\psi) \subset (\frac{1}{2},1)$ and $\int\psi(r) \d r = 1$. Then \begin{align*}
    \Psi(x) = \int_0^{|x|}\psi(r)\d r,\quad x\in\mathbb{R}^n,
\end{align*} is $\Psi$ is the smooth cut-off function satisfying \cref{Cut-offFunction}.

By the property of $\Psi$, $s(x)\Psi\left(\frac{x}{h}\right)$ can be continuously extend to the origin by letting $s(0)\Psi(0) = \lim_{x\to 0}s(x)\Psi\left(\frac{x}{h}\right)=0$. Therefore, for arbitrary continuous function $f$ with compact support, we have \begin{align}
    T_h^0[f] & = h^n\sum_{\beta\in\mathbb{Z}^n\setminus\{0\}} f(\beta h)  \\
    & = h^n\sum_{\beta\in\mathbb{Z}^n\setminus\{0\}} f(\beta h)\Psi(\beta) + h^n \underbrace{f(0)\Psi(0)}_{=0} = T_h\left[f(\cdot)\Psi\left(\frac{\cdot}{h}\right) \right].
\end{align} Hence we can split \begin{align} \label{SplitbyPsi}
    \int f(x)\d x - T_h^0[f] & =  \int f(x)\left(1-\Psi\left(\frac{x}{h}\right)\right)\d x \nonumber \\
    & + \int f(x)\Psi\left(\frac{x}{h}\right)\d x - T_h\left[f(\cdot) \Psi\left(\frac{\cdot}{h}\right)\right].
\end{align} We now state and prove \cref{lemma:ZeroLemmaErrorBound,,lemma:FirstLemmainErrorBound} and \cref{theorem:Second Main Thm}, which are preliminary results for the existence of the limit $\lim_{h\to 0}c(h)$ in \cref{RHSofLinearSystem}.
\begin{lemma}\label{lemma:ZeroLemmaErrorBound}
Let $k\in\N_0$, for any integer $1\leq j\leq n$, there exist a constant $A > 0$, depending on $j \text{ and } k$ such that \begin{align}
    |\partial_j^k s(x)|\leq \frac{A}{|x|^{n-\delta+k}}.\nonumber
\end{align}
\end{lemma}
\begin{proof}
By induction. Fix $j$. By the dilation property \cref{eq:dilation cond}, there exists a smooth function $\rho_0$ on $\R^n$ such that $s(x) = |x|^{\delta - n}\rho_0\left(\frac{x}{|x|}\right)$ for all $x\in\R^n\setminus\{0\}$. Hence \begin{align*}
    |\partial_j^0s(x)| = |s(x)|\leq |x|^{\delta - n}\sup_{x\in S^{n - 1}}|\rho_0(x)|.
\end{align*} Let $k > 0$ and assume there exists a smooth function $\rho_k$ on $\R^n$ such that \begin{align}
    \partial_j^k s(x) = |x|^{\delta - n - k}\rho_k\left(\frac{x}{|x|}\right),\quad \forall x\in\R^n\setminus\{0\}. \label{eq:Induction hypo}
\end{align} Then, we have \begin{align}
    \partial_j^{k + 1} s(x) & = \partial_j\left\{|x|^{\delta - n - k}\rho_k\left(\frac{x}{|x|}\right)\right\} \nonumber \\
    & = |x|^{\delta-n-k-1}\left[(\delta - n - k)\frac{x_j}{|x|} + \left(\partial_j \rho_k\left(\frac{x}{|x|}\right) -\sum_{i = 1}^n\partial_i \rho_k\left(\frac{x}{|x|}\right)\frac{x_ix_j}{|x|^2} \right) \right] \nonumber \\
    & \coloneqq |x|^{\delta - n - k - 1}\rho_{k + 1}\left(\frac{x}{|x|}\right).\nonumber
\end{align} It is easy to see that $\rho_{k + 1}(x) = (\delta - n - k)x_j + \partial_j \rho_k(x) -\sum_{i = 1}^n\partial_i \rho_k(x)x_ix_j$ is a smooth function on $\R^n$. Therefore, the induction hypothesis \cref{eq:Induction hypo} is true for all $k\in\N_0$ and \begin{align}
    |\partial_j^k s(x)| \leq |x|^{\delta - n - k} \sup_{x\in S^{n-1}}|\rho_k(x)|, \quad \forall k \in \N_0.\nonumber
\end{align}
\end{proof}

\begin{lemma} \label{lemma:FirstLemmainErrorBound}
Let $\eta\in \mathbb{N}_0^n$ and $k\in\N_0$, for any integer $1\leq j\leq n$, there exists a constant $A>0$, depending on $j,k \text{ and }\eta$ such that \begin{align}\label{ErrorBoundLemma1}
    \left|\partial^{k}_j\big(s(x)\ x^{\eta}\big)\right| \leq \frac{A}{|x|^{n-\delta+k-|\eta|_1}},\quad \forall x\in\R^n\setminus\{0\}.
\end{align}
\begin{proof}
By Leibniz's product rule, it suffices to show for any $0\leq j\leq n$ and $0\leq l\leq k$, there exists $A>0$ such that \begin{align*}
    |\partial_j^{k-l}s(x)\ \partial_j^{l}x^\eta| \leq \frac{A}{|x|^{n-\delta+k-|\eta|_1}}.
\end{align*} From calculus we have \begin{align*}
    \partial_j^lx^\eta = \begin{cases}
    \frac{\eta_j!}{(\eta_j - l)!} x^{\eta-le_j}, & l\leq \eta_j \\
    0, &\text{ else}
    \end{cases}.
\end{align*} Let $\xi$ be an arbitrary multi-index and denote $M_\xi$ to be the supremum of the function $x\rightarrow |x^\xi|$ on the unit sphere $S^{n-1}$, then $|x^\xi|\leq M_\xi |x|^\xi$ for all $x\in \R^n$. Hence, by \cref{lemma:ZeroLemmaErrorBound} we have \begin{align*}
    |\partial_j^{k-l}s(x)\ \partial_j^{l}x^\eta|&\leq \frac{C_1}{|x|^{n-\delta+k-l}} \cdot C_2|x|^{|\eta|_1-l} \leq \frac{A}{|x|^{n-\delta+k-|\eta|_1}}.
\end{align*}
\end{proof}
\end{lemma}
\begin{theorem} \label{theorem:Second Main Thm}
Let $n\in\mathbb{N}$, $p\in\mathbb{N}_0$ and $\xi$ be a fixed multi-index in $\N_0^n$. Assume $g\in C_c^{N}(\mathbb{R}^n):N> \max\{p+1+\delta+|\xi|_1,n\}$ such that $\partial^{k}g(0) = 0$ for all multi-indices $|k|_1=1,\dots,p$. Then \begin{align}
    \left| \int_{\mathbb{R}^n} g\cdot s \cdot x^{\xi}\d x - T_h^0[g\cdot s\cdot x^{\xi}] - g(0)h^{|\xi|_1+\delta}c(\xi)\right|\leq C h^{p+1 +\delta+ |\xi|_1},
\end{align} where $c(\xi)\in\mathbb{R}$ and the constant $C$ depends only on $g$ and $p$.
\end{theorem}
\begin{proof}
    We write \begin{align*}
        s_{\xi}(x) = s(x)x^{\xi} \quad \text{and}\quad f = g\cdot s_{\xi},
    \end{align*} By the property of $\Psi$, the first term of \cref{SplitbyPsi} can be computed by \begin{align}\label{Splitterm1}
    & \int f(x)\left(1-\Psi\left(\frac{x}{h}\right)\right) \d x = \int_{B(0,h)} g(x)s_{\xi}(x)\left(1-\Psi\left(\frac{x}{h}\right)\right) \d x \nonumber \\
    & = h^{n}\int_{|x|\leq 1} g(hx)s_{\xi}(hx)\left(1-\Psi(x)\right) \d x 
    = h^{|\xi|_1+\delta}\int_{|x|\leq 1} g(hx)s_{\xi}(x)\left(1-\Psi(x)\right) \d x.
\end{align} We have used the dilation property of $s$ in the last equality. From the assumptions on $g$ and $p+\delta+|\xi|_1 > 0$, Taylor's theorem and $n$-dimensional spherical co-ordinate transform, with radial direction denoted by $r$, we have \begin{align} \label{Splitterm2}
    &\left|\int_{|x|\leq 1}(g(hx)-g(0))s_{\xi}(x)(1-\Psi(x))\d x \right| \nonumber \\
    &= \left|\int_{|x|\leq 1}\left(g(hx)-\sum_{\substack{k\in\mathbb{N}_0^n,\\|k|_1\leq p}}\frac{\partial^{k}g(0)}{k!\partial x^{k}}(hx)^{k}\right)s_{\xi}(x)(1-\Psi(x))\d x\right|  \nonumber\\
    &\leq h^{p+1} \sum_{\substack{k\in\mathbb{N}_0^n,\\|k|_1= p+1}}\int_{|x|\leq 1}\left|\frac{\partial^{k}g(\rho)}{k!\partial x^{k}}\right||x|^{p+1+|\xi|_1-n+\delta}|1-\Psi(x)| \d x  \nonumber \\
    &\leq C \max_{|k|_1=p+1}\|\partial^{k}g\|_{\infty}\ h^{p+1}\int_0^1 r^{p+|\xi|_1+\delta}\ \d r \leq  C\ h^{p+1},
\end{align} where $\rho$ is a point on the line between $hx$ and $0$. Combining \cref{Splitterm1} and \cref{Splitterm2}, we have \begin{align} \label{FirstPartEstimate}
    \left|\int f(x)\left(1-\Psi\left(\frac{x}{h}\right)\right)\d x - g(0)h^{|\xi|_1+\delta}\int_{|x|\leq 1}s_{\xi}(x)(1-\Psi(x))\d x\right| 
    \leq C h^{p+1+\delta+|\xi|_1}.
    \end{align} We define a dilation operator $(\tau^a\theta)(x)\coloneqq \theta(ax)$ where $a>0$ and $\theta\in C(\R^n)$. Noting that $f_{\Psi,h}\coloneqq f(\cdot)\Psi(\frac{\cdot}{h})\in C_c^{N}(\mathbb{R}^n)$, then so is $\tau^hf_{\Psi,h}$. Hence, we apply Poisson Summation formula to $\tau^hf_{\Psi,h}$ and get
    \begin{align} \label{PoissonSummationFormula2}
        \sum_{\beta\in\mathbb{Z}^n}(\tau^hf_{\Psi,h})(\beta) 
         = \sum_{k\in\mathbb{Z}^n}\widehat{\tau^hf_{\Psi,h}}(k),
    \end{align} thus \begin{align}\label{DecompositionOfT_h}
        T_h[f_{\Psi,h}] &= T_h\left[f(\cdot)\Psi(\frac{\cdot}{h})\right]  
         = h^n\sum_{\beta\in\mathbb{Z}^n}f(h\beta)\Psi(\beta)  
         = h^n\sum_{\beta\in\mathbb{Z}^n}(\tau^hf_{\Psi,h})(\beta)  \nonumber \\
        & = h^n \sum_{k\in\mathbb{Z}^n}\widehat{\tau^hf_{\Psi,h}}(k) 
         = h^n\sum_{k\in\mathbb{Z}^n} h^{-n}\widehat{f_{\Psi,h}}\left(\frac{k}{h}\right) \nonumber \\
        & = \int_{\mathbb{R}^n} f(x)\Psi\left(\frac{x}{h}\right)\d x + \sum_{k\in\mathbb{Z}^n\setminus\{0\}}\widehat{f_{\Psi,h}}\left(\frac{k}{h}\right).
    \end{align}
    Denoting \begin{align*}
        I_{\xi}(h,k)\coloneqq \widehat{f_{\Psi,h}}\left(\frac{k}{h}\right),
    \end{align*} from \cref{PoissonSummationFormula2} and \cref{DecompositionOfT_h}, we know 
    \begin{align}
        \sum_{k\in\mathbb{Z}^n\setminus\{0\}} |I_{\xi}(h,k)| &< \infty, \\
        T_h[f_{\Psi,h}] - \int_{\mathbb{R}^n} f(x)\Psi\left(\frac{x}{h}\right) \ \d x & = \sum_{k\in\mathbb{Z}^n\setminus\{0\}} I_{\xi}(h,k). \label{TrapzIntegralDiff}
    \end{align}
    The following claim provides an error estimate for each $I_{\xi}(h,k)$. 
    \begin{claim}\label{TechnicalLemma}
Assume all conditions stated in \Cref{theorem:Second Main Thm}. For each $k\in\mathbb{Z}^n\setminus\{0\}$, let $k_j=\max_{l=1,\cdots,n}|k_l|$.
There exists a function $W(\xi,k)$, independent of $h \text{ and } g$, and a constant $C$ that depends on $g$ and $N$ such that \begin{align}
    \left|I_{\xi}(h,k) - g(0)W(\xi,k)h^{|\xi|_1+\delta}(2\pi k_j)^{-N}\right|\leq C |k|^{-N}h^{p+1+\delta + |\xi|_1}.
\end{align}
\end{claim} \begin{proof}[Proof of \cref{TechnicalLemma}]
     By using $\exp(i\lambda z) = \frac{1}{(i\lambda)^N}\frac{\d\empty^N}{\d z^N}(\exp(i\lambda z)),\ \lambda,z\in\mathbb{R}\setminus\{0\}$ and integration by parts repeatedly, \begin{align}
        I_{\xi}(h,k) &= \int_{\mathbb{R}^n}\exp\left(-\frac{2\pi i}{h}k\cdot x\right) g(x)s_{\xi}(x)\Psi\left(\frac{x}{h}\right)\d x \nonumber \\
        &= \left(\frac{-ih}{2\pi k_j}\right)^{N}\int_{\mathbb{R}^n}\partial^N_j\left(\Psi\left(\frac{x}{h}\right)g(x)s_{\xi}(x)\right)\exp\left(\frac{-2\pi i}{h}k\cdot x\right)\d x.
    \end{align} Here, $i$ is the imaginary unit and $\partial_j$ means taking derivative with respect to $x_j$. Define \begin{align}
        W(\xi,k)\coloneqq (-i)^{N}\int_{\mathbb{R}^n}\partial^N_j\left(\Psi \cdot s_{\xi}\right)(x)\exp\left(-2\pi i k \cdot x\right)\d x.
    \end{align} We note that $W$ does not depend on $h$ or $g$ and $\partial_j^{l}\Psi\in C_c^{\infty}$ for any $l>0$. By the condition on $N$ we know $N-(|\xi|_1 + \delta) > 1$. It follows from \cref{lemma:FirstLemmainErrorBound} that \begin{align}
        \int_{|x|>1}|\partial^{N}_j s_{\xi}(x)|\d x &\leq C \int_{|x|>1}\frac{1}{|x|^{n-\delta - |\xi|_1 + N}}\d x
        \leq C\int_1^{\infty}r^{-1+\delta+|\xi|_1-N}\d r <\infty.
    \end{align} Therefore, \begin{align}\label{EstimateW(k)}
        |W(\xi,k)| \leq C\left(\int_{\mathbb{R}^n}\frac{\Psi(x)}{|x|^{n-\delta-|\xi|_1+N}}\d x +\sum_{0<l\leq N}\int_{\mathbb{R}^n}\partial^{l}_j \Psi(x)\partial^{N-l}_j s_{\xi}(x)\d x \right)<\infty.
    \end{align} The estimate \cref{EstimateW(k)} shows the boundedness of $W$ is independent of $k$. By using dilation property of $s$, we re-scale the integral \begin{align}\label{RescaledW(k)}
        W(\xi,k)  &= (-i)^{N} h^{N-n}\int_{\mathbb{R}^n}\partial^{N}_j\left(\Psi\left(\frac{x}{h}\right)s_{\xi}\left(\frac{x}{h}\right)\right)\exp\left(\frac{-2\pi i}{h}k \cdot x\right)\d x  \\
         &= (-i)^{N} h^{N-\delta-|\xi|_1}\int_{\mathbb{R}^n}\partial^{N}_j\left(\Psi\left(\frac{x}{h}\right)s_{\xi}(x)\right)\exp\left(\frac{-2\pi i}{h}k\cdot x\right)\d x.
    \end{align} Then, denote $r(x)\coloneqq \Psi(\sfrac{x}{h})(g(x)-g(0))$, from \cref{RescaledW(k)}
    \begin{align} \label{EstI_xi}
        &I_{\xi}(h,k) - g(0)W(\xi,k)h^{|\xi|_1+\delta}(2\pi k_j)^{-N}  \nonumber \\
        &= \left(\frac{-i h}{2\pi k_j}\right)^{N} \int_{\mathbb{R}^n}\partial^{N}_j \left\{r(x)s_{\xi}(x)\right\}\exp\left(\frac{-2\pi i}{h}k\cdot x\right)\d x. 
    \end{align} It can be seen that for all $0\leq m\leq N$, \begin{align} \label{eq:est partial_r}
    \partial^{m}_jr(x) = \begin{cases}
        0 ,& 0\leq |x|\leq \frac{h}{2} \\
        \sum_{0\leq l\leq m}\binom{m}{l}\frac{1}{h^{m-l}}\partial^{m-l}_j\Psi\left(\frac{x}{h}\right)(\partial^{l}_jg(x)-\delta_{0,l}g(0)), & \frac{h}{2}\leq |x|\leq h \\
        \partial^{m}_jg(x)-\delta_{0,m}g(0),& |x|>h
    \end{cases}.
        % \partial^{m}_jr(x) = \left\{ \begin{array}{ccc}
        %     0 &,& 0\leq |x|\leq \frac{h}{2}\\
        %     \sum_{0\leq l\leq m}\binom{m}{l}\frac{1}{h^{m-l}}\partial^{m-l}_j\Psi\left(\frac{x}{h}\right)(\partial^{l}_jg(x)-\delta_{0,l}g(0)) &,& \frac{h}{2}\leq |x|\leq h \\
        %     \partial^{m}_jg(x)-\delta_{0,m}g(0) &,& |x|>h
        % \end{array}\right. .
    \end{align} Here we used Kronecker delta $\delta_{k_1,k_2}=\mathbbm{1}_{k_1}(k_2)$. Substituting \cref{eq:est partial_r} into RHS of \cref{EstI_xi}, we have \begin{align}\label{eq:partial^j[r(x) - g(0)Psi(x/h)]}
        & \int_{\mathbb{R}^n}\partial^{m}_jr(x)\partial_j^{N-m}s_{\xi}(x)\exp\left(\frac{-2\pi i}{h} k\cdot x\right)\d x \nonumber \\
        & = \sum_{0\leq l\leq m}\frac{\binom{m}{l}}{h^{m-l}} \int_{\frac{h}{2}\leq |x|\leq h}\partial^{m-l}_j \Psi\left(\frac{x}{h}\right)\Big[\partial^{l}_jg(x)-\delta_{0,l}g(0)\Big]\partial^{N-m}_j s_{\xi}(x)\exp\left(\frac{-2\pi i}{h} k\cdot x\right)\d x  \nonumber \\
        & + \int_{|x|>h}\Big[\partial^{m}_jg(x) - \delta_{0,m}g(0)\Big]\partial^{N-m}_j s_{\xi}(x)\exp\left(\frac{-2\pi i}{h} k\cdot x\right)\d x.
    \end{align} Let $T_{m,1},T_{m,2}$ be the first and second terms in RHS of \cref{eq:partial^j[r(x) - g(0)Psi(x/h)]} respectively. Since $(\partial^{l}_jg)(0)=0$ for all $l=1,\cdots,p$, we have for all $l\leq p$ \begin{align}
        |\partial^{l}_j g(x)-(\partial^{l}_jg)(0)|\leq C \max_{\beta\in\mathbb{N}_0^n:|\beta|_1=p+1}\|\partial^{\beta}g\|_{\infty}|x|^{p+1-l}.
    \end{align} By uniform boundedness of $\partial_j^lg(x)$ for all $p<l\leq N$, it follows that for all $0\leq m\leq N$ \begin{align} \label{eq:partial_deriv_g_estimate}
        |\partial^{m}_jg(x)-\delta_{0,m}(\partial^{m}_jg)(0)|\leq C |x|^{\max\{p+1-m,0\}}.
    \end{align} For all $0\leq m\leq N$ and $0\leq l\leq m$, we estimate each term in $T_{m,1}$, from \cref{eq:partial_deriv_g_estimate} and \cref{lemma:FirstLemmainErrorBound}, \begin{align}\label{eq:FirstEstinExtraConvOrder}
        & \frac{1}{h^{m-l}}\left|\int_{\frac{h}{2}\leq |x|\leq h}\partial^{m-l}_j \Psi\left(\frac{x}{h}\right)\Big[\partial^{l}_j g(x)-\delta_{0,l}g(0)\Big](\partial^{N-m}_j s_{\xi})(x)\exp\left(\frac{-2\pi i}{h} k\cdot x\right)\d x\right|  \nonumber \\
        & = \frac{1}{h^{m-l}}h^n\left|\int_{\frac{1}{2}\leq |x|\leq 1}\partial^{m-l}_j \Psi\left(x\right)\Big[\partial^{l}_j g(hx)-\delta_{0,l}g(0)\Big](\partial^{N-m}_j s_{\xi})(hx)\exp\left(-2\pi i k\cdot x\right)\d x\right|  \nonumber \\
        % & \leq \frac{C}{h^{m-l}}\int_{\frac{h}{2}\leq |x|\leq h} |x|^{-(n-2+\alpha-2|\xi|_1+M-m)+\max\{\hat{p}+1-l,0\}}\d x  \nonumber \\
        & \leq ch^{n-(m-l)} \int_{\frac{1}{2}\leq |x|\leq 1}|h x|^{-(n-\delta-|\xi|_1+N-m)+\max\{p+1-l,0\}}\d x  \nonumber\\
        & \leq C h^{p+1+\delta+|\xi|_1-N}\int_{\frac{1}{2}}^1 r^{p+\delta+|\xi|_1+m+l-N}\d r  
         \leq C h^{p+1+\delta+|\xi|_1-N}, 
    \end{align} where we have used $h^{\max\{\gamma,0\}}\leq h^{\gamma}$ if $0<h\leq 1$. It follows that $|T_{m,1}|\leq C h^{p+1+\delta+|\xi|_1-N}$.
    To estimate $T_{m,2}$, we firstly consider the case where $m = 0$, then \begin{align}\label{eq:SecondEstinExtraConvOrder 1}
        |T_{0,2}| & = \left|\int_{|x|>h}[g(x)-g(0)](\partial^{N}_j s_{\xi})(x)\exp\left(\frac{-2\pi i}{h} k\cdot x\right)\d x\right|  \nonumber \\
        & \leq \int_{h<|x|\leq 1}\left|[g(x)-g(0)]\partial^{N}_j s_{\xi}(x)\right|\d x + \int_{|x|>1}\left|[g(x)-g(0)]\partial^{N}_j s_{\xi}(x)\right|\d x \nonumber\\
        &\leq C\int_{h<|x|\leq 1}\frac{|x|^{p+1}}{|x|^{n-\delta-|\xi|_1+N}}\d x + 2\sup_{x\in\R^n}|g(x)|C\int_{|x|>1}\frac{1}{|x|^{n-\delta-|\xi|_1+N}}\d x \nonumber \\
        &\leq C\left( \int_{h}^1 r^{p+\delta+|\xi|_1-N}\d r + \int_{1}^\infty r^{-1+\delta+|\xi|_1-N}\d r \right) \nonumber \\
        &\leq C\left( \frac{h^{p+1+\delta+|\xi|_1-N}-1}{N-(p+1+\delta+|\xi|_1)} + \frac{1}{N-(\delta+|\xi|_1)}\right) \leq C h^{p+1+\delta+|\xi|_1-N}
    \end{align} where we have used $N > p+\delta+|\xi|_1+1$ in the second-to-last inequality. For any continuous function with compact support $u\in C_c(\R^n)$ and for any $\zeta\in \N$, we also have $(1 + |x|)^\zeta u\in C_c(\R^n)$, so there exists $A_\zeta > 0$ such that $|u(x)|\leq \frac{A_\zeta}{(1 + |x|)^\zeta}$ for all $x\in\R^n$.  We now consider the case where $m\not=0$, \begin{align}\label{eq:SecondEstinExtraConvOrder 2}
        |T_{m,2}| & = \left|\int_{|x|>h}\partial^{m}_jg(x)\partial^{M-m}_j s_{\xi}(x)\exp\left(\frac{-2\pi i}{h} k\cdot x\right)\d x\right|  \nonumber \\
        & \leq \int_{h<|x|\leq 1}|\partial^m_jg(x) \partial^{N-m}_js_{\xi}(x)|\d x + \int_{|x| > 1}|\partial_j^mg(x)\partial^{N-m}_js_{\xi}(x)|\d x \nonumber\\
        & \leq C\left(\int_{h<|x|\leq 1} \frac{|x|^{\max\{p+1-m,0\}}}{|x|^{n-\delta-|\xi|_1+N-m}}\d x + \int_{|x| > 1}\frac{1}{|x|^m}\frac{\d x}{|x|^{n-\delta+N-m-|\xi|_1}}\right)\nonumber \\
        & \leq C\left(\int_h^1 r^{p+\delta+|\xi|_1-N}\d r + \int_1^\infty r^{-1+\delta+|\xi|_1 - N}\d r\right) \leq Ch^{p+1+\delta+|\xi|_1-N}.
        % |T_{m,2}| & = \left|\int_{|x|>h}\partial^{m}_jg(x)\partial^{M-m}_j s_{\xi}(x)\exp\left(\frac{-2\pi i}{h} k\cdot x\right)\d x\right|  \nonumber \\
        % & \leq \int_{h<|x|\leq 1}|\partial^m_jg(x) \partial^{N-m}_js_{\xi}(x)|\d x + \int_{1<|x|\leq R}|\partial_j^mg(x)\partial^{N-m}_js_{\xi}(x)|\d x \nonumber\\
        % & \leq C\left(\int_{h<|x|\leq 1} \frac{|x|^{\max\{p+1-m,0\}}}{|x|^{n-\delta-|\xi|_1+N-m}}\d x + \int_{1<|x|\leq R}\frac{\d x}{|x|^{n-\delta+N-m-|\xi|_1}}\right)\nonumber \\
        % & \leq C\left(\int_h^1 r^{p+\delta+|\xi|_1-N}\d r + \int_1^R r^{-1+\delta+|\xi|_1}\d r\right)\nonumber \\
        % & \leq Ch^{p+1+\delta+|\xi|_1-N}.
    \end{align} 
    % The last line follows from $h^{p+1+\delta+|\xi|_1-N} > 1$. 
    Putting together \crefrange{eq:FirstEstinExtraConvOrder}{eq:SecondEstinExtraConvOrder 2}  and using $|k|\leq \sqrt{n}|k_j|$, we obtain from \cref{EstI_xi} \begin{align}
        &|I_{\xi}(h,k)-g(0)W(\xi,k)h^{|\xi|_1+\delta}(2\pi k_j)^{-N}| \nonumber \\
        &\leq C\left(\frac{h}{|k_j|}\right)^N\sum_{m=0}^N (|T_{m,1}| + |T_{m,2}|) \nonumber\\
        &\leq C \left(\frac{h}{|k_j|}\right)^{N}h^{p+1+\delta+|\xi|_1-N} \nonumber \\
        &\leq C|k|^{-N}h^{p+1+\delta+|\xi|_1}. \end{align}
\end{proof} 

Returning to the proof of \cref{theorem:Second Main Thm}, we define \begin{align}
    c(\xi) &= c_1(\xi) - c_2(\xi), \label{WeightsAll}\\
    \text{where} \quad c_1(\xi) &\coloneqq \int_{|x|\leq 1}s_{\xi}(x)(1-\Psi(x))\d x<\infty, \label{Weights1}\\
    c_2(\xi) &\coloneqq \sum_{\substack{k\in\mathbb{Z}^n\setminus\{0\},\\ k_j = \max_l|k_l|}}W(\xi,k)(2\pi k_j)^{-N}<\infty. \label{Weights2}
\end{align} \Cref{Weights2} holds since for each $k\in\mathbb{Z}^n\setminus\{0\}$, $|k|\leq \sqrt{n}|k_j|$ and so \begin{align*}
    \sum_{k\not= 0}|k_j|^{-N}\lesssim \sum_{k\not = 0}|k|^{-N}\leq \sum_{k\not = 0}|k|^{-(n+1)}<\infty.
\end{align*} Using \cref{SplitbyPsi,FirstPartEstimate,DecompositionOfT_h,TrapzIntegralDiff}, \crefrange{WeightsAll}{Weights2} and \Cref{TechnicalLemma}, we obtain \begin{align}
    & \left|\int_{\mathbb{R}^n} f(x)\d x - T_h^0[f] - g(0)h^{\delta+|\xi|_1}c(\xi)\right|  \nonumber \\
    & \leq \left|\int_{\mathbb{R}^n} f(x)\Big(1-\Psi\Big(\frac{x}{h}\Big)\Big)\ \d x - g(0)h^{\delta+|\xi|_1}c_1(\xi)\right| \nonumber \\
    & + \left|\int_{\mathbb{R}^n} f(x)\Psi\Big(\frac{x}{h}\Big)\d x - T_h\Big[f(\cdot)\Psi\Big(\frac{\cdot}{h}\Big)\Big] + g(0)h^{\delta + |\xi|_1}c_2(\xi)\right|  \nonumber \\
    & \overset{\cref{TrapzIntegralDiff}}{\leq} \left|\int_{\mathbb{R}^n} f(x)\Big(1-\Psi\Big(\frac{x}{h}\Big)\Big)d x - g(0)h^{\delta+|\xi|_1}c_1(\xi)\right| 
     + \left|g(0)h^{\delta+|\xi|_1}c_2(\xi) -\sum_{k\in\mathbb{Z}^n\setminus\{0\}}I_{\xi}(h,k)\right|  \nonumber \\
    & \overset{\cref{FirstPartEstimate},\cref{Weights2}}{\leq} Ch^{p+1+\delta+|\xi|_1} + \sum_{\substack{k\in\mathbb{Z}^n,k\not = 0,\\ k_j = \max_l|k_l|}}|I_{\xi}(h,k) - g(0)h^{\delta+|\xi|_1}W(\xi,k)(2\pi k_j)^{-N}| \nonumber \\
    & \leq  Ch^{p+1+\delta+|\xi|_1} + Ch^{p+1+\delta+|\xi|_1}\sum_{\substack{k\in\mathbb{Z}^n,k\not = 0, \\k_j = \max_l|k_l|}}|k_j|^{-N} \nonumber \\
    &\leq  Ch^{p+1+\delta+|\xi|_1}.
\end{align} This concludes the proof of \cref{theorem:Second Main Thm}.
\end{proof} 
\begin{remark}
If we would like to get rid of the compact support criteria of $g$ in \cref{theorem:Second Main Thm}, an alternative is assuming $g\in \mathcal{S}(\R^n)$, the space of Schwartz functions. As one can check, the proof runs almost identically.
\end{remark}
The following two straightforward corollaries are used in proving \cref{theorem:MainTheorem}.
\begin{corollary}\label{corollary:good function}
Let $p\in\N_0$. Assume either $g\in C_c^{N}(\mathbb{R}^n):N>\max\{2p+1+\delta,n\}$ or $g\in\mathcal{S}(\R^n)$ such that $\partial^{k}g(0) = 0$ for all $|k|_1=1,\dots,p$. For any fixed $\xi\in\mathbb{N}_0^n$ with $|\xi|_1\leq p$, we have \begin{align}
    \left| \int_{\mathbb{R}^n} g\cdot s \cdot x^{\xi}\ \d x - T_h^0[g\cdot s\cdot x^{\xi}] - g(0)h^{|\xi|_1+\delta}c(\xi)\right|\leq C h^{p+1+\delta+|\xi|_1},
\end{align} where $c(\xi)\in\mathbb{R}$ and the constant $C$ depends only on $g$ and $p$.
\end{corollary}
\begin{corollary}\label{corollary:estimate of trapez alone}
Let $p\in\N_0$. Assume either $g\in C_c^N(\mathbb{R}^n)$ such that $N> \max\{p+1+\delta,n\}$ or $g\in\mathcal{S}(\R^n)$ and $(\partial^kg)(0) = 0$ for all $k\in\mathbb{N}_0^n:|k|_1\leq p$. Then \begin{align*}
    \left|\int_{\mathbb{R}^n}g\cdot s\d x - T_h^0[g\cdot s]\right|\leq Ch^{p+1+\delta}.
\end{align*}
\end{corollary}

We are ready to prove the main theorem \cref{theorem:MainTheorem}.
\begin{proof}[Proof of \cref{theorem:MainTheorem}]
We first show the limit of the solution $\omega(h)$ of the linear system \cref{MatrixLinearSystem} exists, i.e. $\lim_{h\to 0}\omega(h) = \bar{\omega}$. To this end, we apply \cref{corollary:good function} to each $c_i(h)$ in \cref{RHSofLinearSystem}, yielding \begin{align}
        c_i(h) & = \frac{1}{h^{2|\xi_i|_1-\kappa+\delta}}\left(\int g(x)s(x)x^{2\xi_i-\sum_{j=1}^{\kappa}e_j}\ \d x - T_h^0[g\cdot s\cdot x^{2\xi_i-\sum_{j=1}^{\kappa}e_j}]\right) \nonumber \\
        & = \frac{1}{h^{2|\xi_i|_1-\kappa+\delta}} \left(g(0)h^{2|\xi_i|_1-\kappa+\delta}c(\xi_i) + O(h^{2p-\kappa+2+\delta+2|\xi_i|_1-\kappa})\right) \nonumber \\
        & = c(\xi_i) + O(h^{2p-\kappa+2})\underset{h\to 0}{\longrightarrow} c(\xi_i). \label{Conv_of_c(h)}
    \end{align} Here we write $c(2\xi_i-\sum_{j=1}^{\kappa}e_j)$ as $c(\xi_i)$ for convenience. Define \begin{align}
        \mathfrak{C} = (c(\xi_1),\dots,c(\xi_{I_{n,p}}))^T. \label{eq:frak C}
    \end{align} It is clear that \begin{align}
        \boldsymbol{G}(h) = \text{diag}(g(\xi_1h),\dots,g(\xi_{I_{n,p}}h)))\underset{h\to 0}{\longrightarrow} \boldsymbol{Id}_{I_{n,p}\times I_{n,p}},
    \end{align} and so \begin{align}
        \boldsymbol{G}(h)^{-1} = \text{diag}(g(\xi_1h)^{-1},\dots,g(\xi_{I_{n,p}}h)^{-1}))\underset{h\to 0}{\longrightarrow} \boldsymbol{Id}_{I_{n,p}\times I_{n,p}},
    \end{align} where $\boldsymbol{Id}_{m\times m}$ denotes the identity matrix of size $m\times m$. In addition, $\|\boldsymbol{G}(h)^{-1}\|_2 = \max\{|g(\xi_kh)|^{-1}:1\leq k\leq I_{n,p}\}$ is bounded in a neighbourhood of origin. Define $\bar{\omega}$ to be the solution of the linear system \begin{align}\label{MatrixLinearSystem3}
        \boldsymbol{K}\bar{\omega} = \mathfrak{C}.
    \end{align} 
    Assuming $\boldsymbol{K}$ is non-singular, which will be proved in next section, from \cref{MatrixLinearSystem} and \cref{MatrixLinearSystem3}, we obtain \begin{align}
        |\omega(h) - \bar{\omega}|_2 & = |\boldsymbol{G}(h)^{-1}\boldsymbol{K}^{-1}c(h) - \boldsymbol{K}^{-1}\mathfrak{C}|_2 \nonumber \\
        & = |\boldsymbol{G}(h)^{-1}(\boldsymbol{Id}_{I_{n,p}\times I_{n,p}}-\boldsymbol{G}(h))\boldsymbol{K}^{-1}c(h) + \boldsymbol{K}^{-1}(c(h)-\mathfrak{C})|_2 \nonumber \\
        & \leq \|\boldsymbol{G}(h)^{-1}\|_2 \ \|\boldsymbol{Id}_{I_{n,p}\times I_{n,p}}-\boldsymbol{G}(h)\|_2\ \|\boldsymbol{K}^{-1}\|_2 \ |c(h)|_2 + \|\boldsymbol{K}^{-1}\|_2\ |c(h)-\mathfrak{C}|_2. \label{Conv_Omega1}
    \end{align} A Taylor expansion of $g$ around the origin gives \begin{align}
        |1- g(\xi_k h)|\leq Ch^{2p-\kappa+2}, \quad k=1,\dots,I_{n,p}, 
    \end{align} and hence \begin{align}
        \|\boldsymbol{Id}_{I_{n,p}\times I_{n,p}} - \boldsymbol{G}(h)\|_2 = \max\{|1-g(\xi_k h)|:k=1,\dots,I_{n,p}\} \leq Ch^{2p-\kappa+2}. \label{Conv_I_minus_G(h)}
    \end{align} From \cref{Conv_of_c(h),,Conv_Omega1,,Conv_I_minus_G(h)}, we therefore have \begin{align}\label{Conv_omega(h)}
        |\omega(h) - \bar{\omega}|_2\leq Ch^{2p-\kappa+2},
    \end{align} where the constant $C$ depends only on $g$ and $p$.
    
    We now prove the accuracy of the corrected trapezoidal rule \cref{AccuracyofCorrectedTrapezRule}. We denote the Taylor polynomial of $\phi$ at the origin by \begin{align}
        P_{\phi}(x)=\sum_{\substack{\xi\in\mathbb{N}_0^n,\\ 0\leq |\xi|_1\leq 2p-\kappa+1}}\frac{x^{\xi}}{\xi!}\frac{\partial^{\xi}\phi(0)}{\partial x^{\xi}},
    \end{align} and define $\tilde{\phi}=P_{\phi}\cdot g$. It is clear that \begin{align}\label{PhisTaylor}
    \partial^{k}[\phi(x)-\tilde{\phi}(x)]|_{x=0}=0,\quad \forall\ k\in\mathbb{N}_0^n:|k|_1=0,1,\dots,2p-\kappa+1.
\end{align} Writing \begin{align*}
    f\coloneqq \phi\cdot s,\quad\text{and }\quad I[f]\coloneqq \int_{\mathbb{R}^n\setminus\{0\}} f(x)\d x,
\end{align*} we split $Q_h^p[f]-I[f]$ as \begin{align}
    Q_h^p[f]-I[f] & = A_h^p[\phi-\tilde{\phi}]+(T_h^0-I)[(\phi-\tilde{\phi})\cdot s]+ (Q_h^p-I)[\tilde{\phi}\cdot s] \nonumber \\
    & \eqqcolon E_1+E_2+E_3 .
\end{align} To estimate $E_1$, from \cref{PhisTaylor} and Taylor's theorem we have \begin{align*}
    |\phi(x)-\tilde{\phi}(x)|\leq \max_{\substack{x\in\mathbb{R}^n,\\ k\in\mathbb{N}_0^n,\\|k|_1=2p-\kappa+1}}\left(\frac{|\partial^k(\phi - \tilde{\phi})(x)|}{k!}\right)\ |x|^{2p-\kappa+2}.
\end{align*} Then from \cref{A^p} we get \begin{align}\label{E1Est}
    |E_1| = |A_h^p[\phi-\tilde{\phi}]| 
    \leq Ch^{\delta}\max_{1\leq j\leq I_{n,p}}|\bar{\omega}_j|\max_{\xi\in\mathcal{G}_j:1\leq j\leq I_{n,p}}|\phi(\xi h)-\tilde{\phi}(\xi h)| 
    \leq Ch^{2p-\kappa+2+\delta}.
\end{align} In order to bound $E_2$, we apply \cref{corollary:estimate of trapez alone} to $\phi-\tilde{\phi}$, giving rise to \begin{align}\label{E2Est}
    |(T_h^0-I)[(\phi-\tilde{\phi})\cdot s]|\leq Ch^{2p-\kappa+2+\delta}.
\end{align} Finally, it remains to estimate $E_3$. Since \begin{align}\label{PreE_3}
    |E_3|&= |Q^p_h[\tilde{\phi}\cdot s] - I[\tilde{\phi}\cdot s]| \nonumber \\
    &\leq \sum_{\substack{\xi\in\mathbb{N}_0^n,\\0\leq|\xi|_1\leq 2p-\kappa+1}}\frac{1}{\xi!}\left|\frac{\partial^{\xi}\phi(0)}{\partial x^{\xi}}\right||Q_h^p[g\cdot s\cdot x^{\xi}] - I[g\cdot s\cdot x^{\xi}]| \nonumber \\
    &\leq C\sum_{\substack{\xi\in\mathbb{N}_0^n,\\0\leq|\xi|_1\leq 2p-\kappa+1}}|T_h^0[g\cdot s\cdot x^{\xi}] - I[g\cdot s\cdot x^{\xi}]+A^p_h[g\cdot x^{\xi}]|.
\end{align} Let $\xi = (\xi_1,\dots,\xi_n)$, we claim that $T_h^0[g\cdot s\cdot x^{\xi}]$, $I[g\cdot s\cdot x^{\xi}]$ and $A^p_h[g\cdot x^{\xi}]$ vanish if either at least one of $\xi_j$ is even, $1\leq j\leq \kappa $ or at least one of $\xi_j$ is odd, $\kappa < j\leq n $. To see this, by symmetry of $g$ and $s$,  \begin{align} \label{OddpowerI}
    I[g\cdot s\cdot x^{\xi}] & = \int_{\mathbb{R}^n}g(x)s(x)x^{\xi} \d x \nonumber \\
    & = \int_{[0,\infty)^n}g(x)s(x)x^{\xi}\prod_{j = 1}^\kappa\Big[1 - (-1)^{\xi_j}\Big]\prod_{j = \kappa+1}^n\Big[1 + (-1)^{\xi_j}\Big] \d x.
\end{align} It is easy to see that \cref{OddpowerI} vanish if either at least one of $\xi_j$ is even, $1\leq j\leq \kappa $ or at least one of $\xi_j$ is odd, $\kappa< j\leq n$. The argument for $T_h^0[g\cdot s\cdot x^{\xi}]=0$ under same condition for $\xi$ closely resembles above. Recall that $g(\beta h)$ is constant over each $\beta\in\mathcal{G}_\eta$. Denote $g_\eta = g(\beta h),\ \forall \beta\in\mathcal{G}_\eta$. By \cref{A^p} and definition of the set $\mathcal{G}_\eta$ at \cref{eq:G_eta}, we have \begin{align}
    A^p_h[g\cdot x^{\xi}]  = h^\delta\sum_{\eta\in\mathcal{I}(n,p)}\bar{\omega}_{\eta} g_{\eta} \sum_{\beta\in\mathcal{G}_{\eta}}\text{sgn}\left(\prod_{j=1}^{\kappa}\beta_j\right)(\beta h)^{\xi}\, \nonumber
\end{align} and \begin{align}
    \sum_{\beta\in\mathcal{G}_{\eta}}\text{sgn}\left(\prod_{j=1}^{\kappa}\beta_j\right)(\beta h)^{\xi} & = h^{|\xi|_1}\sum_{\beta\in\mathcal{G}_\eta}\text{sgn}\left(\prod_{j=1}^{\kappa}\beta_j\right) \beta^\xi \nonumber \\
    & = h^{|\xi|_1}\eta^\xi \sum_{i = 1}^n\sum_{k_i\in\{0,1\}} \text{sgn}\left( (-1)^{\sum_{j = 1}^\kappa k_j} \right)(-1)^{\sum_{j = 1}^n k_j\xi_j} \nonumber \\
    & = h^{|\xi|_1}\eta^\xi\sum_{i = 1}^n\sum_{k_i\in\{0,1\}} (-1)^{\sum_{j = 1}^\kappa k_j(\xi_j + 1) + \sum_{j = \kappa + 1}^n k_j\xi_j} \nonumber \\
    & = h^{|\xi|_1}\eta^\xi\prod_{j=1}^{\kappa}\big[1 - (-1)^{\xi_j} \big]\prod_{j=\kappa+1}^n\big[1 + (-1)^{\xi_j} \big]. \label{eq:Ap discrete sum}
\end{align}\Cref{eq:Ap discrete sum} will vanish if either at least one of $\xi_j$ is even, $1\leq j\leq \kappa $ or at least one of $\xi_j$ is odd, $\kappa< j\leq n $. Therefore, we can re-write \cref{PreE_3} as \begin{align}\label{Pre2E3}
    |E_3| &\leq C\sum_{|\xi|_1 = \kappa,\kappa+2,\cdots,2p-\kappa}|T_h^0[g\cdot s\cdot x^{\xi}] - I[g\cdot s\cdot x^{\xi}]+A^p_h[g\cdot x^{\xi}]| \nonumber \\
    &= C\sum_{\xi\in \mathcal{I}(n,p)}|T_h^0[g\cdot s\cdot x^{2\xi-\sum_{j=1}^{\kappa}e_j}] - I[g\cdot s\cdot x^{2\xi-\sum_{j=1}^{\kappa}e_j}]+A^p_h[g\cdot x^{2\xi-\sum_{j=1}^{\kappa}e_j}]|.
\end{align} By \cref{SecondCorrectedTrapezRule,,Conv_omega(h)}, $E_3$ can be estimated by \begin{align}\label{E3Est}
    |E_3|& \leq C h^\delta \sum_{\eta\in\mathcal{I}(n,p)}|\omega_\eta-\bar{\omega}_\eta|\left(\sum_{\beta\in\mathcal{G}_\eta}|g(\beta h)||\beta h|^{2|\xi|_1-\sum_{j=1}^{\kappa}e_j}\right) \nonumber \\
    &= Ch^{\delta}|\omega(h) - \bar{\omega}|_2 
    \leq Ch^{2p-\kappa+2+\delta}.
\end{align}
Putting together \cref{E1Est,,E2Est,,E3Est}, we have proved \begin{align}
     |Q_h^p[f] - I[f]|\leq Ch^{2p-\kappa+2+\delta}.
\end{align}
\end{proof} 

The following corollary provides a Taylor's expansion of $c(h)$ in \cref{RHSofLinearSystem}. It is used in Richardson extrapolation for numerical computation of correction weights. Recall that $\mathfrak{C} = \lim_{h\to 0}c(h)$ in \cref{eq:frak C}.
\begin{corollary}\label{corollary:Taylor for C}
Let $p,k\in\mathbb{N}_0:2p\geq \kappa$. For any $g\in \mathcal{S}(\mathbb{R}^n)$ such that $g(0) = 1$, $\partial^lg(0) = 0,\ l\in\N_0^n$ with $0\not = |l|_1\leq 2p-\kappa+1$, we have\begin{align}
    c(h) = \mathfrak{C} + h^{2p-\kappa+2}\sum_{j=0}^k A_jh^{2j},
\end{align} where $A_j$'s are constants independent of $h$.
\end{corollary}
\begin{proof}
    If $k=0$, this is proved in the  \Cref{theorem:MainTheorem}. Assume corollary holds for some $k\in\mathbb{N}$. We prove the corollary holds for $k+1$. Fix $\xi_i\in\mathcal{I}(n,p)$. It is suffices to show \begin{align*}
        c_i(h) = C(\xi_i) + h^{2p-\kappa+2}\sum_{j=0}^{k+1}A_j h^{2j}.
    \end{align*} To this end, let $g^*\in \mathcal{S}(\R^n)$ such that $g^*(0) = 2$ and $\partial^lg^*(0) = 0,\ l\in\N_0^n:0\not = |l|_1\leq 2(p+k)+3$, then by  \Cref{theorem:Second Main Thm} we have \begin{align}\label{eq:Formula1_small_corollary_on}
        \Big|\frac{1}{h^{2|\xi_i|_1-\kappa+\delta}}\Big[&\int_{\mathbb{R}^n}g^*(x)s(x)x^{2\xi_i-\sum_{j=1}^{\kappa}e_j}\d x \nonumber \\
        & - T_h^0[g^*\cdot s\cdot x^{2\xi_i-\sum_{j=1}^{\kappa}e_j}] \Big] - 2C(\xi_i) \Big|\lesssim h^{2p+2+2k+2}.
    \end{align} Let $\bar{g} = g^* - g$, then $\bar{g}(0) = 1,\ \partial^l \bar{g}(0) = 0$ for all $l\in\mathbb{N}^2_0:0\not = |l|_1\leq 2p-\kappa+1$. Applying inductive hypothesis to $\bar{g}$ we have \begin{align}\label{eq:Formula2_small_corollary_on}
        \Big|&\frac{1}{h^{2|\xi_i|_1-\kappa+\delta}}\Big[\int_{\mathbb{R}^n}\bar{g}(x)s(x)x^{2\xi_i-\sum_{j=1}^{\kappa}e_j}\d x \nonumber \\
        & - T_h^0[\bar{g}\cdot s\cdot x^{2\xi_i-\sum_{j=1}^{\kappa}e_j}] \Big]  - C(\xi_i)\Big|\leq h^{2p-\kappa+2}O\left(\sum_{j=0}^k h^{2j}\right).
    \end{align} Summing \cref{eq:Formula1_small_corollary_on} with \cref{eq:Formula2_small_corollary_on}, 
    \begin{align}
        |c_i(h) - C(\xi_i)| & =\left| \frac{1}{h^{2|\xi_i|_1-\kappa+\delta}}\left[\int_{\mathbb{R}^n}g(x)s(x)x^{2\xi_i-\sum_{j=1}^{\kappa}e_j}\d x - T_h^0[g\cdot s\cdot x^{2\xi_i-\sum_{j=1}^{\kappa}e_j}]\right] - C(\xi_i)\right| \nonumber \\
        & \leq h^{2p-\kappa+2}O\left(\sum_{j=0}^{k+1} h^{2j}\right). \nonumber
    \end{align} This completes the induction.
\end{proof}

%%%%%%%%%%%%%%%%%%%%%%%%%%%%%%%%%%%%%%%%%%%%%%%%%%%%%%%%%%%%%%%%%%%%%%%%%%%%%%%%%%%%%%%%%%%%%%%%%%%%%%%%%%%%%%%%%
\section{Non-singularity of the Coefficient Matrix \texorpdfstring{$\boldsymbol{K}$}{TEXT}} \label{section:Non-singularity}
\subsection{Case of \texorpdfstring{$\kappa = 0$}{TEXT}}\label{subsection:Non-singularity I}
In this section, we assume $\kappa=0$ in the symmetry condition \cref{eq:symmetry} of weakly singular kernel $s$ and we prove the non-singularity of the corresponding coefficient matrix $\boldsymbol{K}$ defined in \cref{MatrixK} for arbitrary $p\in\mathbb{N}$ in $n$ dimensions. It turns out that the $\kappa\not = 0$ cases immediately follows the $\kappa=0$ case, which will be briefly outlined in next section as a result. We introduce some notations, definitions and preliminary results in the first place. 
% \begin{notation} Let $n,p\in \N$,
% \begin{align}
%     \mathcal{I}(n,p) &\coloneqq \{\xi\in\mathbb{N}_0^n:|\xi|_1\leq p\}, \\
%     \mathcal{I}^+(n,p) &\coloneqq \{\xi\in\mathbb{N}^n:|\xi|_1\leq p\},\\
%     \mathcal{L}(n,p) &\coloneqq \{\xi\in\mathbb{N}_0^n:|\xi|_1=p\},\\
%     \mathcal{L}^+(n,p) &\coloneqq \{\xi\in\mathbb{N}^n:|\xi|_1=p\}.
% \end{align}
% \end{notation}
\begin{definition}
Let  $m\in \N$ and $(A_i)_{i=1}^m$ be a sequence of mutually disjoint sets, we write the their union $B = \cup_{i=1}^m A_i$ as \begin{align*}
    B = \bigsqcup_{i=1}^m A_i.
\end{align*}
\end{definition}
\begin{definition}
Let $m_1,m_2\in\mathbb{N}$, the \emph{multiplicity counting functions} over $\xi\in\N_0^{m_1}$ and $\mathcal{S}\subset \N^{m_2}$ are defined by\begin{align}
    \lambda(\xi,j) &\coloneqq \sum_{i=1}^{m_1}\mathbbm{1}_{j}(\xi_i),\quad \xi = (\xi_1,\dots,\xi_{m_1})\in\mathbb{N}_0^{m_1},\ j\in\mathbb{N}_0 \label{def:mul_count_1}\\
    \Lambda_{\mathcal{S}}(j) &\coloneqq \sum_{\xi\in \mathcal{S}}\lambda(\xi,j),\ j\in\N \label{def:mul_count_2}.
\end{align}
\end{definition}
\begin{remark}
We present a easy fact of \cref{def:mul_count_2} that will be used repeatedly later without explicit mentioning. Let $m\in N$ and $\mathcal{S}_1,\mathcal{S}_2\subset \N^{m}$ such that $\mathcal{S}_1 \cap \mathcal{S}_2 = \varnothing$, then $\Lambda_{\mathcal{S}_1\sqcup \mathcal{S}_2} = \Lambda_{\mathcal{S}_1} + \Lambda_{\mathcal{S}_2}$.
\end{remark}
We tabulated all frequently appeared notations in this section in \cref{tab:notations on}.
\begin{table}[h!]
    \centering
    \begin{tabular}{c|c|c}
    \hline\hline
    \textbf{notation} & \textbf{definition} & comment\\
    \hline
        $\mathcal{I}(n,p)$ & $\{\xi\in\mathbb{N}_0^n:|\xi|_1\leq p\}$ &  $n\in \N,\ p\in\N_0$ \\
        $\mathcal{I}^+(n,p)$ & $\{\xi\in\mathbb{N}^n:|\xi|_1\leq p\}$ &  $n\in \N,\ p\in\N_0$\\
        $\mathcal{L}(n,p)$ & $\{\xi\in\mathbb{N}_0^n:|\xi|_1=p\}$ & $n\in \N,\ p\in\N_0$\\
        $\mathcal{L}^+(n,p)$ & $\{\xi\in\mathbb{N}^n:|\xi|_1=p\}$ & $n\in \N,\ p\in\N_0$\\
        \hline
        $\mathcal{I}(n,p;i,k)$ & $\{\xi\in\mathcal{I}(n,p):\lambda(\xi,i) = k\}$ & $n\in \N,\ p,i,k\in\N_0$ \\
        $\mathcal{I}_J(n,p;i,k)$ & $\{\xi\in\mathcal{I}(n,p;i,k):\xi_j = i \ \forall j\in J,\ \xi_j\not = i\ \forall j\not\in J\}$ & $J\subset\{1,\cdots,n\}:\ |J|=k$\\
        \hline
        $\mathcal{L}(n,p;i,k)$ & $\{\xi\in\mathcal{L}(n,p):\lambda(\xi,i) = k\}$ & $n\in \N,\ p,i,k\in\N_0$\\
        $\mathcal{L}_J(n,p;i,k)$ & $\{\xi\in\mathcal{L}(n,p;i,k):\xi_j = i \ \forall j\in J,\ \xi_j\not = i\ \forall j\not\in J\}$ & $J\subset\{1,\cdots,n\}:\ |J|=k$\\
        \hline
        \multirow{2}{1em}{$J_k$} & \multirow{2}{10em}{$\{J_{k,i}:i=1,\dots, {\binom{n}{k}}\}$} & $k\in \N_0,\ i=1,\cdots,\binom{n}{k}$ \\
        & & $J_{k,i}\subset\{1,\dots,n\}:|J_{k,i}|=k$ \\
        \hline
        $\mathcal{N}(m,p)$ & $|\mathcal{L}^+(m,p)|$ & $m,p\in\N_0$, $\mathcal{N}(0,p)\coloneqq \mathbbm{1}_0(p)$\\
        \hline\hline
    \end{tabular}
    \caption{List of notations used in \cref{subsection:Non-singularity I}}
    \label{tab:notations on}
\end{table}

% \begin{notation}
% Let $n,p\in\mathbb{N}$ and $i,k\in\N_0$, we denote \begin{align*}
%     \mathcal{I}(n,p;i,k)&\coloneqq \{\xi\in\mathcal{I}(n,p):\lambda(\xi,i) = k\},\\
%     \mathcal{L}(n,p;i,k)&\coloneqq \{\xi\in\mathcal{L}(n,p):\lambda(\xi,i) = k\}.
% \end{align*} For any $J\subset\{1,\dots,n\}:|J|=k$, we let \begin{align*}
%     \mathcal{I}_J(n,p;i,k)&\coloneqq\{\xi\in\mathcal{I}(n,p;i,k):\xi_j = i \ \forall j\in J,\ \xi_j\not = i\ \forall j\not\in J\},\\
%     \mathcal{L}_J(n,p;i,k)&\coloneqq\{\xi\in\mathcal{L}(n,p;i,k):\xi_j = i \ \forall j\in J,\ \xi_j\not = i\ \forall j\not\in J\}.
% \end{align*} 
% We denote $J_k$ to be the collection of all $J\subset\{1,\dots,n\}:|J|=k$, so $J_k = \{J_{k,i}:i=1,\dots, {n\choose k}\}$.
% \end{notation} 
\begin{remark}
It is easy to see that there are $\binom{n}{k}$'s distinct $J$ such that $J\subset\{1,\cdots,n\}$ and $|J|=k$ and there exist bijections between $\mathcal{I}_J(n,p;i,k) \cong \mathcal{I}^+(n-k,p-ki)$ and $\mathcal{L}_J(n,p;i,k) \cong \mathcal{L}^+(n-k,p-ki)$ for each $J$. We always consider these bijections being canonical, i.e., given $J$, a projection defined by sending $\xi$ to $\xi'\coloneqq (\xi_{i_1},\dots,\xi_{i_{n-k}}):i_1,\dots,i_{n-k}\not \in J$. 
\end{remark}
We now specify the index on $\mathcal{I}(n,p)$ based on the mutually disjoint decomposition \begin{align}
    \mathcal{I}(n,p) = \bigsqcup_{k=0}^{n}\mathcal{I}(n,p;0,k) 
     = \{0\}\sqcup \bigsqcup_{k=1}^{n}\bigsqcup_{i=1}^{\binom{n}{k}}\mathcal{I}_{J_{k,i}}(n,p;0,k). \nonumber
\end{align} Note that $\mathcal{I}(n,p;0,0) = \mathcal{I}^+(n,p)$ and $\mathcal{I}(n,p;0,n) = \{0\}$. For each $1\leq k\leq n$ and $1\leq i\leq \binom{n}{k}$, $\mathcal{I}_{J_{k,i}}(n,m;0,k)$ is indexed by the dictionary order, the index for $J_k = (J_{k,i})_{i}$ is arbitrary for each $k$. We then list all the elements in $\mathcal{I}(n,p;0,k)$ from $\mathcal{I}_{J_{k,1}}(n,p;0,k)$ to $\mathcal{I}_{J_{k,{n\choose k}}}(n,p;0,k)$. We finally list all the elements in $\mathcal{I}(n,p)$ by $0$, followed by the elements in $\mathcal{I}(n,p;0,n-1),\ \cdots$ and followed by $\mathcal{I}^+(n,p)$.
\begin{definition}\label{definition: generated matrix}
A matrix $\boldsymbol{M}$ is said to be generated by indexed finite sets $\mathcal{S}_1$ and $\mathcal{S}_2$, where $\mathcal{S}_1, \mathcal{S}_2\subset \N_0^n$ if, for each $1\leq i\leq |\mathcal{S}_1|$ and $1\leq j\leq |\mathcal{S}_2|$, \begin{align*}
    M_{i,j} = \eta_{j}^{2\xi_{i}},
\end{align*} where $\xi_{i},\eta_{j}$ are $i$-th and $j$-th elements in $\mathcal{S}_1$ and $\mathcal{S}_2$, respectively. If in addition $\mathcal{S}_1=\mathcal{S}_2=\mathcal{S}$, $\boldsymbol{M}$ is said to be generated by $\mathcal{S}$. Moreover, if $S_1=\varnothing$ or $S_2=\varnothing$, $\boldsymbol{M}$ does not exist.
\end{definition}
We arrived at the first major theorem concerning the structure of $\boldsymbol{K}$.
\begin{theorem}\label{theorem:main structure}
Let $n,p\in\N$, the coefficient matrix $\boldsymbol{K}$ defined in \cref{MatrixK} is a upper-triangular block matrix with square sub-blocks on diagonal \begin{align*}
    \boldsymbol{K} = \begin{pmatrix}
    \boldsymbol{A}_{n} & \bigstar & \bigstar & \bigstar & \bigstar \\
     & 2\boldsymbol{A}_{n-1} & \bigstar & \bigstar & \bigstar \\
     &  & \ddots & \bigstar & \bigstar \\
     &  &  & 2^{n-1}\boldsymbol{A}_1 & \bigstar \\
     &  &  &  & 2^n\boldsymbol{A}_0
    \end{pmatrix} .
\end{align*} For each $0\leq k\leq n$, $\boldsymbol{A}_k$ is generated by $\mathcal{I}(n,p;0,k)$. Moreover, for all $0<k<n$, $\boldsymbol{A}_k$ is block diagonal matrix of the form \begin{align*}
    \boldsymbol{A}_k =\begin{pmatrix}
    \boldsymbol{B}_{k,1} &  &  \\
     & \ddots & \\
     &  & & \boldsymbol{B}_{k,{n\choose k}}
    \end{pmatrix}
\end{align*} where each sub-sub-block $\boldsymbol{B}_{k,l}$ is generated by $\mathcal{I}_{J_{k,l}}(n,p;0,k)$.
\end{theorem}
\begin{remark}
It is not necessary that all $\boldsymbol{A}_k,\ 1\leq k\leq n$ exist, we ignore those non-existent $\boldsymbol{A}_k$. In fact, if $p<n$, then $\boldsymbol{A}_0,\cdots,\boldsymbol{A}_{n-p-1}$ do not exist. Otherwise, all $\boldsymbol{A}_k$ exist.
\end{remark}
\begin{proof}[Proof of \cref{theorem:main structure}]
Let us first prove that $\boldsymbol{A}_k$ has the required form. Note that $\mathcal{I}(n,p;0,n) = \{0\}$, so $\boldsymbol{A}_n = 1$ a singleton. Fix $0<k<n$, let $1\leq l_1,l_2\leq {n\choose k}:l_1\not = l_2$ and denote $\mathcal{S}_1 = \mathcal{I}_{J_{k,l_1}}(n,p;0,k)$ and $\mathcal{S}_2 = \mathcal{I}_{J_{k,l_2}}(n,p;0,k)$. By $\mathcal{S}_1,\mathcal{S}_1\subset \{1,\dots,n\}$, $|\mathcal{S}_1|=|\mathcal{S}_2|$ and $\mathcal{S}_1 \not = \mathcal{S}_2$ , it is easy to check that $|\mathcal{S}_1\bigtriangleup\mathcal{S}_2|\geq 2$, therefore we only need to show the matrix $\boldsymbol{M}$ generated by $\mathcal{S}_1$ and $\mathcal{S}_2$ is $\boldsymbol{0}$. We assume without loss of generality that $1\in J_{k,l_2}\setminus J_{k,l_1}$, then for all $\xi\in\mathcal{S}_1$, $\xi_1 \not = 0$ and for all $\eta\in\mathcal{S}_2$, $\eta_1 = 0$. As a result, \begin{align*}
    \eta^{2\xi} = \underbrace{\eta_1^{2\xi_1}}_{=0} \prod_{j>1}\eta_j^{2\xi_j} = 0,
\end{align*} however this is general element in $\boldsymbol{M}$. Hence $\boldsymbol{M} = \boldsymbol{0}$. 

We proceed to show $\boldsymbol{K}$ has the required form. By construction of $\boldsymbol{K}$ in \cref{MatrixK}, a general element in $\boldsymbol{K}$ is of the form $C(\eta)\eta^{2\xi}$ where $\eta,\xi\in \mathcal{I}(n,p)$ and $C(\eta)=2^{n-\lambda(\eta,0)}$, hence $\boldsymbol{K}$ can be decomposed into  sub-matrices generated by $\mathcal{I}(n,p;0,k)$ and $\mathcal{I}(n,p;0,l)$ where $0\leq k,l\leq n$ with each being multiplied by some factors of $2$. Fix $k < l$, denote $\mathcal{S}_k = \mathcal{I}(n,p;0,k)$ and $\mathcal{S}_l = \mathcal{I}(n,p;0,l)$, we are left to show the matrix $\boldsymbol{M}$ generated by $\mathcal{S}_k$ and $\mathcal{S}_l$ is $\boldsymbol{0}$, this gives rise to upper-triangular block structure of $\boldsymbol{K}$. To see this, let $\xi\in \mathcal{S}_k$ and $\eta\in\mathcal{S}_l$, so the general element in $\boldsymbol{M}$ is $\eta^{2\xi}$. Note that $\lambda(\eta,0) = l>k=\lambda(\xi,0)$, we pick $j\in\{1,\cdots,n\}$ such that $\eta_j = 0$ while $\xi_j\not = 0$, therefore \begin{align*}
    \eta^{2\xi} = \underbrace{\eta_j^{2\xi_j}}_{=0}\prod_{i\not=j}\eta_i^{2\xi_i} = 0.
\end{align*} The proof is complete.
\end{proof}
The following corollary is immediate from \cref{theorem:main structure}.
\begin{corollary}\label{corollary:Cor_detK}
Let $n,p\in\mathbb{N}$,
\begin{align} 
    \det{\boldsymbol{K}} & = C\prod_{k = 0}^{n-1}\det{\boldsymbol{A}_k} = C\det{\boldsymbol{A}_0}\prod_{k = 1}^{n-1}\prod_{j = 1}^{n\choose k}\det{\boldsymbol{B}_{k,j}}. \nonumber
\end{align} where $C$ is a non-zero constant. Note that if any $\boldsymbol{A}_k$ does not exists, we take its determinant as $1$ to preserve the structure of the formula.
\end{corollary}
\begin{proof}
This is well-known properties of the block diagonal and block triangular matrices, for a reference see paragraphs 0.9.2 and 0.9.4 in \cite{horn2013matrix}.
\end{proof}
The \cref{corollary:Cor_detK} tells us that non-singularity of $\boldsymbol{K}$ is determined by matrices $\boldsymbol{A}_0$ and $\boldsymbol{B}_{k,j}$'s. The next lemma reveals we need to focus on matrices $\boldsymbol{D}_m$ generated by $\mathcal{I}^+(m,p)$ where $m\in\N:m\leq n$
. Furthermore, we indeed only have to study matrix $\boldsymbol{D}_n$, since $n$ is arbitrary.
\begin{lemma}\label{lemma:equivalence of B and D}
Let $n,p\in\N$, let $1\leq k\leq n-1$ and $1\leq j\leq {n\choose k}$, matrices $\boldsymbol{D}_m$ are generated by $\mathcal{I}^+(m,p)$, $m\in\N:m\leq n$. Then
\begin{align*}
    \boldsymbol{B}_{k,j} = \boldsymbol{D}_{n-k}.
\end{align*} Moreover, $\boldsymbol{A}_0 = \boldsymbol{D}_n$.
\end{lemma}
\begin{proof}
$\boldsymbol{A}_0 = \boldsymbol{D}_n$ is trivial. Fix $k,j$, recall that there exists a canonical bijection $\Psi$ between $\mathcal{I}_{J_j}(n,p;0,k)$ and $\mathcal{I}^+(n-k,p)$, It is not hard to show that $\Psi$ preserves relative order and hence indexing. Assume without loss of generaliy that $J_j = \{1,\cdots,k\}$, then for all $\xi\in \mathcal{I}_{J_j}(n,p;0,k)$, we have $\xi_l = 0$ for all $l\in J$ while $\xi_l\not = 0$ for all $l\not \in J_j$. Let $\eta,\xi\in \mathcal{I}_{J_j}(n,p;0,k)$, we can see that \begin{align*}
    \eta^{2\xi} = \prod_{l=k+1}^n\eta_j^{2\xi_j} = \Psi(\eta)^{2\Psi(\xi) }.
\end{align*} The RHS of above equation is a general element in $\boldsymbol{D}_{n-k}$ that has same position with $\eta^{2\xi}$ in $\boldsymbol{B}_{k,j}$.
\end{proof} By \cref{lemma:equivalence of B and D}, \cref{corollary:Cor_detK} becomes 
\begin{corollary}\label{corollary:detK v2}
Let $n,p\in\mathbb{N}$,
\begin{align} 
    \det{\boldsymbol{K}} & = C\prod_{k = 1}^{n}\big(\det{\boldsymbol{D}_k}\big)^{\binom{n}{k}}. \nonumber 
\end{align} where $C$ is a non-zero constant.
\end{corollary}
\begin{proof} Applying \cref{lemma:equivalence of B and D} to \cref{corollary:Cor_detK}, we get 
\begin{align*}
    \det \boldsymbol{K} = C\det \boldsymbol{D}_n\prod_{k=1}^{n-1}(\det \boldsymbol{D}_{n-k} )^{\binom{n}{k}} = C\prod_{k = 1}^{n}\big(\det{\boldsymbol{D}_k}\big)^{\binom{n}{k}}.
\end{align*}
\end{proof}

We are now focusing on the non-singularity of $\boldsymbol{D_n}$. We will introduce some definitions and preliminary results.
\begin{definition}\label{def:falling factorial}
Let $x\in\mathbb{R}$ and $m\in\mathbb{N}_0$, the \emph{falling factorial} is defined to be \begin{align*}
    (x)_0 & \coloneqq 1, \\
    (x)_m & \coloneqq \prod_{k=1}^m (x-k+1)
\end{align*}
\end{definition}
We only need to following elementary property of the falling factorial.
\begin{lemma}\label{lemma:falling factorial}
Let $m,M\in\mathbb{N}$ such that $M > m$, then \begin{align}
    \sum_{j=m}^M(j)_m = \frac{(M+1)_{m+1}}{m+1}. \nonumber
\end{align}
\end{lemma}
\begin{proof}
Let $j,k\in\mathbb{N}$, we note that \begin{align}
    (j+1)_k-(j)_k = (j+1)(j)_{k-1} - (j-k+1)(j)_{k-1} = k(j)_{k-1}. \nonumber
\end{align} Therefore, \begin{align}
    \sum_{j=m}^M(j)_m &= \frac{1}{m+1}\sum_{j=m}^M((j+1)_{m+1}-(j)_{m+1}) \nonumber \\
    & = \frac{1}{m+1}((M+1)_{m+1} - \underbrace{(m)_{m+1}}_{=0}) = \frac{(M+1)_{m+1}}{m+1} \nonumber
\end{align}
\end{proof}
We cite a corollary due to Ruiz \cite[Corollary 2]{10.2307/3618534}.
\begin{corollary}\label{corollary:Ruiz}
For all $n\geq 0$ and $x\in \mathbb{R}$, \begin{align*}
    \sum_{i=0}^n(-1)^i{n\choose i}(x-i)^{n-j} = 0,\quad \forall 1\leq j\leq n.
\end{align*}
\end{corollary}
We establish an algebraic identity.
\begin{lemma}\label{lemma:alg_id}
Let $n\in\N:n\geq 2$, for all $0\leq k\leq n-2$, \begin{align}\label{Alg_id}
    \sum_{m=1}^{n-k}(-1)^m {n\choose m}{n-m\choose k}m = 0.
\end{align}
\end{lemma}
\begin{proof}
We compute \begin{align*}
    LHS & = \sum_{m=1}^{n-k}(-1)^m \frac{n!m}{m!k!(n-m-k)!} \\
    & = \frac{n!}{k!}\sum_{m=1}^{n-k}(-1)^m\frac{1}{m!(n-m-k)!}m = {n\choose k} \sum_{m=1}^{n-k}(-1)^m\frac{(n-k)!}{m!(n-m-k)!}m \\
    & = -{n\choose k} \sum_{m=0}^K(-1)^{m+1}{K\choose m}m. \quad (\text{setting }K = n-k)
\end{align*} Invoking \cref{corollary:Ruiz} and substituting $n$ by $K$, $i$ by $m$ and letting $x = 0$, $j = n-1$, we obtain $LHS = 0$.
\end{proof}

We present a crucial enumeration lemma. 
% \begin{notation}
% For all $m_1,m_2\in\N_0$, we denote the size of $\mathcal{L}^+(m_1,m_2)$ by $\mathcal{N}(m_1,m_2)$. We also agree that $\mathcal{N}(0,m_2)=\mathbbm{1}_0(m_2)$.
% \end{notation}
\begin{lemma} \label{lemma:size of Lplus}
Let $j,m,n\in\N:m>j$,
\begin{equation}
    \Lambda_{\mathcal{L}^+(n,m)}(j) = n\mathcal{N}(n-1,m-j). \nonumber
\end{equation}
\end{lemma}
\begin{proof}
Let $1\leq k \leq n$. We temporarily use the following notation in this proof:\begin{equation*}
    \mathcal{L}^{+}(n,m;j,k) = \{\xi\in\mathcal{L}^+(n,m):\lambda(\xi,j) = k\},
\end{equation*} and if $J\in\{1,\cdots,n\}:|J|=k$, \begin{equation*}
    \mathcal{L}^{+}_J(n,m;j,k) = \{\xi\in\mathcal{L}^{+}(n,m;j,k):\xi_l = j\ \forall l\in J,\xi_l \not= j\ \forall l\not\in J\}.
\end{equation*} It is easy to see that there are $n\choose k$ distinct $J$ and there are bijections between $\mathcal{L}^{+}_J(n,m;j,k)$ and $\mathcal{L}^{+}_{\{1,\cdots,k\}}(n,m;j,k)$ . Therefore, for each fixed $k$, choosing $J = \{1,2,\cdots, k\}$ we have\begin{align}
    \Lambda_{\mathcal{L}^{+}(n,m;j,k)}(j) &= {n\choose k}k\Big|\{\xi\in\N^n:\xi_l = j\ \forall l\in J,\ \xi_l\not = j\ \forall l\not\in J,\ |\xi|_1=m \}\Big| \nonumber \\
    &= {n\choose k}k\Big|\{\xi\in\N^{n-k}:(\xi_1\not=j)\wedge\cdots\wedge(\xi_{n-k}\not= j),\ |\xi|_1=m-kj\} \Big| \nonumber\\
    &= {n\choose k}k\Big|\{\xi\in\N^{n-k}:\neg((\xi_1=j)\vee\cdots\vee(\xi_{n-k}= j)),\ |\xi|_1=m-kj\} \Big| \nonumber\\
    &= {n\choose k}k\Big[\big|\{\xi\in\N^{n-k}:|\xi|_1=m-kj\}\big| \nonumber \\
    & - \big|\{\xi\in\N^{n-k}:(\xi_1=j)\vee\cdots\vee(\xi_{n-k}= j),\ |\xi|_1=m-kj\} \big| \Big]. \label{eq:L^+j}
\end{align} By Exclusion-inclusion principle in naive set theory, we have \begin{align}
    & \Big|\{\xi\in\N^{n-k}:(\xi_1=j)\vee\cdots\vee(\xi_{n-k}= j),\ |\xi|_1=m-kj\} \Big| \nonumber\\
    & = \sum_{l=1}^{n-k}(-1)^{l+1}\sum_{1\leq i_1\leq\cdots\leq i_l\leq n-k}\big|\{\xi\in\N^{n-k}:\xi_{i_1}=\cdots=\xi_{i_l}=j,\ |\xi|_1=m-kj\} \big| \nonumber\\
    & = \sum_{l=1}^{n-k}(-1)^{l+1}{n-k \choose l}\big|\{\xi\in\N^{n-k-l}:|\xi|_1=m-(k+l)j \} \big| \nonumber\\
    & = \sum_{l=1}^{n-k}(-1)^{l+1}{n-k \choose l}\mathcal{N}(n-k-l,m-(k+l)j). \label{eq:a_set_est}
\end{align} Combining \cref{eq:L^+j} and \cref{eq:a_set_est}, we get \begin{equation}
    \Lambda_{\mathcal{L}^{+}(n,m;j,k)}(j) = {n\choose k}k\left(\mathcal{N}(n-k,m-kj) + \sum_{l=1}^{n-k}(-1)^l{n-k\choose l}\mathcal{N}(n-k-l,m-(k+l)j) \right). \nonumber
\end{equation} Using mutually disjoint decomposition \begin{align*}
    \mathcal{L}^+(n,m) = \bigsqcup_{k=0}^{n}\mathcal{L}^{+}(n,m;j,k),
\end{align*} we therefore have, 
\begin{align}
    &\Lambda_{\mathcal{L}^+(n,m)}(j) = \sum_{k=0}^n\Lambda_{\mathcal{L}^{+}(n,m;j,k)}(j) \label{eq:technical sum}\\
    &=0 + \sum_{k=1}^n{n\choose k}k\mathcal{N}(n-k,m-kj) \label{eq:technical sum I}\\
    &+ \sum_{k=1}^n{n\choose k}k\sum_{l=1}^{n-k}(-1)^l{n-k\choose l}\mathcal{N}(n-k-l,m-(k+l)j) \label{eq:technical sum II} 
\end{align} Rearranging the summation, \cref{eq:technical sum I} can be written as \begin{align}
    \sum_{k=1}^n{n\choose k}k\mathcal{N}(n-k,m-kj) = \sum_{k=0}^{n-1}{n\choose k}(n-k)\mathcal{N}(k,m-(n-k)j), \label{eq:technical sum I v2}
\end{align} and \cref{eq:technical sum II} becomes \begin{align}
    &\sum_{k=1}^n{n\choose k}k\sum_{l=1}^{n-k}(-1)^l{n-k\choose l}\mathcal{N}(n-k-l,m-(k+l)j) \nonumber \\
    &= \sum_{t=1}^n\sum_{s=1}^{n-t}(-1)^s{n\choose t}{n-t\choose s}t\ \mathcal{N}(n-t-s,m-(t+s)j)\quad (\text{changing }k\to t,\ l\to s) \nonumber\\
    &= \sum_{t=1}^n\sum_{s'=0}^{n-t-1} (-1)^{n-t-s'}{n\choose t}{n-t\choose n-t-s'}t\ \mathcal{N}(s',m-(n-s')j)\quad (\text{changing }s\to n-t-s') \nonumber\\
    &= \sum_{s'=0}^{n-2}\sum_{t=0}^{n-s'-1} (-1)^{n-t-s'}{n\choose t}{n-t\choose s'}t\ \mathcal{N}(s',m-(n-s')j) \nonumber\\
    &= \sum_{k=0}^{n-2}\sum_{l=0}^{n-k-1} (-1)^{n-k-l}{n\choose l}{n-l\choose k}l\ \mathcal{N}(k,m-(n-k)j).\quad (\text{changing }s'\to k,\ t\to l) \label{eq:technical sum II v2}
\end{align} Therefore, combining \cref{eq:technical sum I v2,,eq:technical sum II v2}, \cref{eq:technical sum} becomes
\begin{align}
    &\Lambda_{\mathcal{L}^+(n,m)}(j) =\sum_{k=0}^{n-1}{n\choose k}(n-k)\mathcal{N}(k,m-(n-k)j) \nonumber \\
    &+\sum_{k=0}^{n-2}(-1)^{n-k}\left[\sum_{l=1}^{n-k-1}(-1)^l{n\choose l}{n-l\choose k}l \right]\mathcal{N}(k,m-(n-k)j) \nonumber \\
    &=n\mathcal{N}(n-1,m-j) + \sum_{k=0}^{n-2}(-1)^{n-k}\underbrace{\left[\sum_{l=1}^{n-k}(-1)^l{n\choose l}{n-l\choose k}l\right]}_{=0}\mathcal{N}(k,m-(n-k)j) \label{eq:vanishing_summation},
\end{align} where \cref{eq:vanishing_summation} holds due to \cref{lemma:alg_id}. As a result, we have \begin{equation}
    \Lambda_{\mathcal{L}^+(n,m)}(j) = n\mathcal{N}(n-1,m-j)
\end{equation} as desired.
\end{proof}
We now give a explicit formula for $\mathcal{N}(n,p)$ for $n\geq 2$. 
\begin{lemma} \label{lemma:formula of N} For all $n\geq 2$ and $p\in\N$,
\begin{equation}
    \mathcal{N}(n,p) = \left\{ \begin{array}{cc}
        C_n(p-1)_{n-1} & p\geq n \\
        0 & \text{otherwise}
    \end{array}\right. ,
\end{equation} where $C_n = \frac{1}{(n-1)!}$.
\end{lemma}
\begin{proof}
If $p<n$, then $\mathcal{L}^+(n,p) = \varnothing$ and $\mathcal{N}(n,p) = 0$. We focus on $p\geq n$. Let $n=2$, $\mathcal{N}(2,p) = p-1$ is clear. We assume the lemma holds for $n$, we now prove the lemma holds for $n+1$. \begin{align}
    \mathcal{N}(n+1,p) &=  |\{\xi\in\N^{n+1}:|\xi|_1=p\}| \nonumber \\
    &= \sum_{l=1}^{p-n}|\{\xi\in\N^{n+1}:\xi_1 = l,\ |\xi|_1=p\}| \nonumber \\
    &= \sum_{j=n}^{p-1}|\{\xi\in\N^n:|\xi|_1=j\}| = \sum_{j=n}^{p-1}\mathcal{N}(n,j) \nonumber \\
    &=C_n\sum_{j=n}^{p-1}(j-1)_{n-1} = C_n\sum_{j=n-1}^{p-2}(j)_{n-1}. \label{eq:N1}
\end{align} By \cref{lemma:falling factorial}, \cref{eq:N1} leads to \begin{align}
    \mathcal{N}(n+1,p) = \frac{C_n}{n}[(p-1)_{n}-0] = C_{n+1}(p-1)_n,
\end{align} with $C_{n+1}=\frac{1}{n!}$.
\end{proof}
\begin{remark}
One can easily check that for all $m\in\N$, $\mathcal{N}(1,m) = 1$.
\end{remark}
We come to the theorem addressing the non-singularity of matrix $\boldsymbol{D}_n$ generated by $\mathcal{I}^+(n,p)$ that is the culmination of all labors. We will use a algebraic argument. To this end, let $x_i\in \R:i = 1,\cdots,p-n+1$, we define a matrix $\boldsymbol{E}\in\R^{|\mathcal{I}^+(n,p)|\times |\mathcal{I}^+(n,p)|}$ by \begin{align}\label{eq:def E}
    E_{i,j}(x_1,\cdots,x_{p-n+1}) = \prod_{l=1}^{n}x_{\xi_{j,l}}^{\xi_{i,l}},
\end{align} here $\xi_i,\xi_j$ are $i$-th and $j$-th elements in $\mathcal{I}^+(n,p)$ respectively.
\begin{theorem}\label{theorem: E}
Let $n,p\in\N:n\geq 2,\ p\geq n$. \begin{align}
    \det{\boldsymbol{E}} = C\prod_{1\leq j\leq p-n+1}x_j^{\Lambda_{\mathcal{I}^+(n,p)}(j)}\prod_{1\leq i< j\leq p-n+1}(x_j-x_i)^{\Lambda_{\mathcal{I}^+(n,p)}(j)},
\end{align} where $C$ is s non-zero constant.
\end{theorem}
\begin{proof}
Denote $F(x_1,\cdots,x_{p-n+1})=\det{\boldsymbol{E}}$ and $I_{n,p}^+ = |\mathcal{I}^+(n,p)|$. Let us first show that for each $1\leq j\leq p-n+1$, $x_j$ divides $F$ and has multiplicity at least $\Lambda_{\mathcal{I}^+(n,p)}(j)$. By construction of $\boldsymbol{E}$, substituting $0$ into $x_j$ we have at least one vanishing column. To see this, note that $(j,1,\cdots,1)\in\mathcal{I}^+(n,p)$, so there is one column such that its $i$-th entry is of the form $x_j^{\xi_{i,1}}x_1^{\sum_{l>1}\xi_{i,l}}$, where $\xi_i$ is the $i$-th element in $\mathcal{I}^+(n,p)$, for all $1\leq i\leq I_{n,p}^+$. This column will vanish if $x_j=0$, we see that $F = 0$ and hence $x_j$ divides $F$. To determine multiplicity, we only need to enumerate the total occurrence of $j$ in the set $\mathcal{I}^+(n,p)$, this is precisely $\Lambda_{\mathcal{I}^+(n,p)}(j)$. 

Similar reasoning applies to proving $x_j-x_i$ divides $F$ with multiplicity at least $\Lambda_{\mathcal{I}^+(n,p)}(j)$ for each pairs of integer $i,j$ such that $1\leq i<j\leq p-n+1$. We claim that $\Lambda_{\mathcal{I}^+(n,p)}(j) \leq \Lambda_{\mathcal{I}^+(n,p)}(i)$. For each $\xi\in\mathcal{I}^+(n,p)$ such that $\lambda(\xi,j) > 0$, replacing all $j$ in $\xi$ by $i$, we obtain a new vector $\eta\in\N^n$ such that $|\eta|_1 < |\xi|_1\leq p$, so $\eta\in\mathcal{I}^+(n,p)$ and the total occurrence of $i$ in $\mathcal{I}^+(n,p)$ is at least as much as that of $j$, i.e., $\Lambda_{\mathcal{I}^+(n,p)}(j) \leq \Lambda_{\mathcal{I}^+(n,p)}(i)$. Now, substituting $x_i$ into $x_j$ we have two identical columns, which implies $x_j-x_i$ divides $F$. Moreover, the multiplicity of $x_j-x_i$ is at least $\Lambda_{\mathcal{I}^+(n,p)}(j)$.

We proceed to determine the degree of $F$. Note that $F$ is a homogeneous multi-variable polynomial. To see this, by construction of $\boldsymbol{E}$, polynomials in each row share same degree, namely $|\xi_i|_1$ for row $i$, where $\xi_i$ is the $i$-th element in $\mathcal{I}^+(n,p)$. By Leibniz formula for determinant \begin{align*}
    F(x_1,\cdots,x_{p-n+1}) = \sum_{\sigma\in S_{I_{n,p}^+}} \text{sgn}(\sigma)\prod_{i=1}^{I_{n,p}^+} E_{i,\sigma(i)},
\end{align*} where $S_m$ is the permutation group on $\{1,\cdots,m\}$, we can deduce that each term in $F$ shares the degree $\sum_{i=1}^{I_{n,p}^+}|\xi_i|_1$. To compute this sum, the mutually disjoint decomposition \begin{align*}
    \mathcal{I}^+(n,p) = \bigsqcup_{k=n}^p\mathcal{L}^+(n,k),
\end{align*} gives rise to \begin{align}
    \deg F = \sum_{k=n}^p k \mathcal{N}(n,k).
\end{align} Using \cref{lemma:falling factorial,,lemma:formula of N}, we have \begin{align}
    \deg F &= \sum_{k=n}^p k\mathcal{N}(n,k) 
     = C_n\sum_{k=n}^p k(k-1)_{n-1} \nonumber \\
    & = C_n\sum_{k=n}^p (k)_n 
     = \frac{C_n}{n+1} (p+1)_{n+1} = \frac{(p+1)_{n+1}}{(n+1)(n-1)!}. 
\end{align} Denote \begin{align}
    H(x_1,\cdots,x_{p-n+1}) = \prod_{1\leq j\leq p-n+1}x_j^{\Lambda_{\mathcal{I}^+(n,p)}(j)}\prod_{1\leq i< j\leq p-n+1}(x_j-x_i)^{\Lambda_{\mathcal{I}^+(n,p)}(j)}.
\end{align} Since $H$ divides $F$, the proof is complete once we show $\deg F = \deg H$. We compute \begin{align}
    \deg H &= \sum_{j=1}^{p-n+1}\Lambda_{\mathcal{I}^+(n,p)}(j) + \sum_{j=2}^{p-n+1}\sum_{i=1}^{j-1}\Lambda_{\mathcal{I}^+(n,p)}(j) \nonumber \\
    &=\sum_{j=1}^{p-n+1}j\Lambda_{\mathcal{I}^+(n,p)}(j). \label{eq:compute deg H}
\end{align} By \cref{lemma:falling factorial,,lemma:size of Lplus,,lemma:formula of N}, we have \begin{align}
    \Lambda_{\mathcal{I}^+(n,p)}(j) &= \sum_{m=j+n-1}^{p}\Lambda_{\mathcal{L}^+(n,m)}(j) \nonumber \\
    &=n\sum_{m=j+n-1}^{p} \mathcal{N}(n-1,m-j) = n \sum_{m=n-1}^{p-j}\mathcal{N}(n-1,m) \nonumber \\
    &=nC_{n-1}\sum_{m=n-1}^{p-j}(m-1)_{n-2} \nonumber\\
    &= \frac{n(p-j)_{n-1}}{(n-1)(n-2)!} = \frac{n}{(n-1)!}(p-j)_{n-1}.
\end{align} Therefore, \begin{align} \label{eq: degH v2}
    \deg H = \frac{n}{(n-1)!}\sum_{j=1}^{p-n+1}j(p-j)_{n-1}.
\end{align} We focus on the summation
\begin{align}
    &\sum_{j=1}^{p-n+1}j(p-j)_{n-1} = -\sum_{j=1}^{p-n+1}[p-j+1 - (p+1)](p-j)_{n-1} \nonumber \\
    &= (p+1)\sum_{j=1}^{p-n+1}(p-j)_{n-1} - \sum_{j=1}^{p-n+1}(p-j+1)_n \nonumber \\
    &=(p+1)\sum_{j=n-1}^{p-1}(j)_{n-1} - \sum_{j=n}^p(j)_n = \frac{(p+1)(p)_n}{n} - \frac{(p+1)_{n+1}}{n+1} = \frac{(p+1)_{n+1}}{n(n+1)}. \label{eq:degH summation}
\end{align} Substituting \cref{eq:degH summation} into \cref{eq: degH v2}, we therefore obtain \begin{align}
    \deg H = \frac{n}{(n-1)!}\frac{(p+1)_{n+1}}{n(n+1)} = \frac{(p+1)_{n+1}}{(n+1)(n-1)!} = \deg F.
\end{align} This concludes the proof.
\end{proof} A immediate consequence of \cref{theorem: E} is $\boldsymbol{D}_{n}$ is non-singular for all $n,p\in\N$.
\begin{corollary} \label{corollary:Dn is nonsingular}
Let $n,p\in\N$, the matrix $\boldsymbol{D}_n$ generated by $\mathcal{I}^+(n,p)$ is non-singular.
\end{corollary}
\begin{proof}
If $p<n$, then $\mathcal{I}^{+}(n,p) = \varnothing$ and $\boldsymbol{D}_n$ does not exists, hence the corollary is vacuously true. Otherwise, let $x_i=i^2$ for all $1\leq i\leq p-n+1$. We note that $\boldsymbol{E}(1^2,\cdots,(p-n+1)^2) = \boldsymbol{D}_n$, according to the definition of $\boldsymbol{E}$ in \cref{eq:def E} and \cref{definition: generated matrix} for $\boldsymbol{D}_n$. Since all $x_i=i^2,\ 1\leq i\leq p-n+1$ are non-zeros and mutually distinct, by \cref{theorem: E}, $\boldsymbol{D}_n$ has non-vanishing determinant, hence the non-singularity.
\end{proof}
Together with \cref{theorem:main structure,,corollary:detK v2,,corollary:Dn is nonsingular}, the following result is now immediate.
\begin{corollary}\label{corollary:K is nonsingular}
Let $n,p\in\mathbb{N}$, the coefficient matrix $\boldsymbol{K}$ defined in \cref{MatrixK} for $\kappa=0$ is non-singular.
\end{corollary}
%%%%%%%%%%%%%%%%%%%%%%%%%%%%%%%%%%%%%%%%%%%%%%%%%%%%%%%%%%%%%%%%%%%%%%%%%%%%%%%%%%%%%%%%%%%%%%%%%%%%%%%%%%%%%%%%%%%%%%%%%%%%%%%%%%%%%%%%%%%%%%%%%%%
\subsection{Case of \texorpdfstring{$\kappa \not= 0$}{TEXT}} \label{subsection:Non-singularity II}
In this section we address the non-singularity of the coefficient matrix $\boldsymbol{K}$ corresponding to $\kappa\not = 0$ cases, for arbitrary $p\in\mathbb{N}$ in $n$ dimensions. Due to similarity with $\kappa=0$ case, this section will be succinct. We use the following notations in this section:
\begin{table}[h!]
    \centering
    \begin{tabular}{c|c|c}
    \hline\hline
    \textbf{notation} & \textbf{definition} & comment\\
    \hline
        $\mathcal{I}(n,p)$ & $\{\xi\in \N^\kappa\times\N_0^{n-\kappa}:|\xi|_1\leq p\}$ &  $n\in \N,\ p\in\N_0$ \\
        $\mathcal{I}^+(n,p)$ & $\{\xi\in \mathcal{I}(n,p):\xi_j>0,\ j=1,\cdots,n\}$ &  $n\in \N,\ p\in\N_0$\\
        $\mathcal{L}(n,p)$ & $\{\xi\in \N^\kappa\times\N_0^{n-\kappa}:|\xi|_1 = p\}$ & $n\in \N,\ p\in\N_0$\\
        $\mathcal{L}^+(n,p)$ & $\{\xi\in \mathcal{L}(n,p):\xi_j>0,\ j=1,\cdots,n\}$ & $n\in \N,\ p\in\N_0$\\
        \hline
        $\mathcal{I}(n,p;0,k)$ & $\{\xi\in\mathcal{I}(n,p):\lambda(\xi,0) = k\}$ & $n\in \N,\ p,k\in\N_0$ \\
        $\mathcal{I}_J(n,p;0,k)$ & $\{\xi\in\mathcal{I}(n,p;0,k):\xi_j = 0 \ \forall j\in J,\ \xi_j\not = 0\ \forall j\not\in J\}$ & $J\subset\{K+1,\cdots,n\}:\ |J|=k$\\
        \hline
        $\mathcal{L}(n,p;0,k)$ & $\{\xi\in\mathcal{L}(n,p):\lambda(\xi,0) = k\}$ & $n\in \N,\ p,k\in\N_0$\\
        $\mathcal{L}_J(n,p;0,k)$ & $\{\xi\in\mathcal{L}(n,p;0,k):\xi_j = 0 \ \forall j\in J,\ \xi_j\not = 0\ \forall j\not\in J\}$ & $J\subset\{K+1,\cdots,n\}:\ |J|=k$\\
        \hline
        \multirow{2}{1em}{$J_k$} & \multirow{2}{12em}{$\{J_{k,i}:i=1,\dots, {\binom{n-\kappa}{k}}\}$} & $k\in \N_0,\ i=1,\cdots,\binom{n-\kappa}{k}$ \\
        & & $J_{k,i}\subset\{K+1,\dots,n\}:|J|=k$ \\
        \hline\hline
    \end{tabular}
    \caption{List of notations used in \cref{subsection:Non-singularity II}}
    \label{tab:notations off}
\end{table}

Given mutually disjoint decomposition of $\mathcal{I}(n,p)$
\begin{align}
    \mathcal{I}(n,p) &= \bigsqcup_{k = 0}^{n-\kappa}\mathcal{I}(n,p;0,k) \label{eq:decompose I_n_p off 1}\\
    &=\bigsqcup_{k = 0}^{n-\kappa}\bigsqcup_{l = 1}^{\binom{n-\kappa}{k}}\mathcal{I}_{J_{k,l}}(n,p;0,k) \label{eq:decompose I_n_p off 2}
\end{align} we can now specify the index for $\mathcal{I}(n,p)$. For each $0\leq k\leq n-\kappa$ and $1\leq l\leq \binom{n-\kappa}{k}$, we index $\mathcal{I}_{J_{k,l}}(n,p;0,k)$ by dictionary order. The index for $J_k = (J_{k,l})_{1\leq l\leq \binom{n-\kappa}{k}}$ is arbitrary. We then list all elements in $\mathcal{I}(n,p;0,k)$ from $\mathcal{I}_{J_{k,1}}(n,p;0,k)$ to $\mathcal{I}_{J_{k,\binom{n-\kappa}{k}}}(n,p;0,k)$. We list all elements in $\mathcal{I}(n,p)$ from $\mathcal{I}(n,p;0,n-\kappa)$ to $\mathcal{I}(n,p;0,0)$.

\begin{theorem}\label{theorem:structure off}
Let $n,p\in\N$, the coefficient matrix $\boldsymbol{K}$ defined in \cref{MatrixK} can be factored into \begin{align}
    \boldsymbol{K} = \boldsymbol{E}\boldsymbol{H},
\end{align} where \begin{enumerate}
    \item $\boldsymbol{H}$ is diagonal matrix with $i$-th entry on diagonal \begin{align*}
    H_{i,i} = \prod_{j=1}^{\kappa}\frac{1}{|\xi_{i,j}|},
    \end{align*} here $\xi_i$ is the $i$-th element in $\mathcal{I}(n,p)$.
    \item $\boldsymbol{E}$ is a upper-triangular block matrix with square blocks on diagonal \begin{align}
    \boldsymbol{E} = \begin{pmatrix}
    2^\kappa\boldsymbol{A}_{n-\kappa} & \bigstar & \bigstar & \bigstar & \bigstar \\
     & 2^{\kappa+1}\boldsymbol{A}_{n-\kappa-1} & \bigstar & \bigstar & \bigstar \\
     &  & \ddots & \bigstar & \bigstar \\
     &  &  & 2^{n-1}\boldsymbol{A}_1 & \bigstar \\
     &  &  &  & 2^n\boldsymbol{A}_0
    \end{pmatrix} .
\end{align}
\end{enumerate}
 For each $0\leq k\leq n-\kappa$, $\boldsymbol{A}_k$ is generated by $\mathcal{I}(n,p;0,k)$. Moreover, for all $0<k<n-\kappa$, $\boldsymbol{A}_k$ is block diagonal matrix of the form \begin{align}
    \boldsymbol{A}_k =\begin{pmatrix}
    \boldsymbol{B}_{k,1} &  &  \\
     & \ddots & \\
     &  & & \boldsymbol{B}_{k,{n-2\choose k}}
    \end{pmatrix}
\end{align} where each sub-block $\boldsymbol{B}_{k,l}$ is generated by $\mathcal{I}_{J_{k,l}}(n,p;0,k)$.
\end{theorem}
\begin{proof}
We first show $\boldsymbol{K}=\boldsymbol{E}\boldsymbol{H}$. To see this, let $\xi_i,\xi_j$ be $i$-th and $j$-th elements in $\mathcal{I}(n,p)$, then by construction of $\boldsymbol{K}$ in \cref{MatrixK}, we have \begin{align}
    K_{i,j} &= \sum_{\beta\in \mathcal{G}_{\xi_j}}\text{sgn}\left(\prod_{j=1}^{\kappa}\beta_j\right)\beta^{2\xi_i-\sum_{j=1}^{\kappa}e_j}  \nonumber \\
    &= \sum_{\beta\in \mathcal{G}_{\xi_j}}\frac{\beta^{2\xi_i}}{\prod_{j=1}^{\kappa}|\beta_j|} = \frac{2^{n-\lambda(\xi_j,0)}}{\prod_{j=1}^{\kappa}|\beta_j|}\xi_j^{2\xi_i}. \label{eq:element of K off}
\end{align} Hence the factorization follows.

We proceed to show the structure of $\boldsymbol{E}$. By \cref{eq:element of K off} and decomposition \cref{eq:decompose I_n_p off 1} we know that $\boldsymbol{E}$ is partitioned by sub-matrices generated by $\mathcal{I}(n,p;0,l)$ and $\mathcal{I}(n,p;0,m)$ where $l\not = m$, multiplied by $2^{n-m}$. It is suffices to show matrix $\boldsymbol{M}$ generated by $\mathcal{I}(n,p;0,l)$ and $\mathcal{I}(n,p;0,m)$ where $l<m$ is $\boldsymbol{0}$. Let $\xi\in \mathcal{I}(n,p;0,l)$ and $\eta\in \mathcal{I}(n,p;0,m)$, by $l<m$ there exists $j\in\{K+1,\cdots,n\}$ such that $\xi_j\not = 0$ but $\eta_j=0$, so \begin{align*}
    \eta^{2\xi} = \underbrace{\eta_j^{2\xi_j}}_{=0}\prod_{i\not=j}\eta_i^{2\xi_i} = 0,
\end{align*} this proves $\boldsymbol{M}=\boldsymbol{0}$ and the upper-triangularity block structure of $\boldsymbol{E}$.

We are left to show $\boldsymbol{A}_k,\ k = 1,\cdots,n-\kappa$ has the required structure. For each $k = 1,\cdots,n-\kappa$, by decomposition \cref{eq:decompose I_n_p off 2} we know that $\boldsymbol{A}_k$ is further partitioned by sub-blocks generated by $\mathcal{I}_{J_{k,l}}(n,p;0,k)$ and $\mathcal{I}_{J_{k,m}}(n,p;0,k)$ where $l\not=m,\ l,m\in\{1,\cdots,\binom{n-\kappa}{k}\}$. Noting that $|J_{k,l}\triangle J_{k,m}|\geq 2$, hence by symmetry it is enough to show matrix $\boldsymbol{M}$ generated by $\mathcal{I}_{J_{k,l}}(n,p;0,k)$ and $\mathcal{I}_{J_{k,m}}(n,p;0,k)$ is $\boldsymbol{0}$. Let $\xi\in \mathcal{I}_{J_{k,l}}(n,p;0,k)$ and $\eta\in \mathcal{I}_{J_{k,m}}(n,p;0,k)$, assume without loss of generality that $\xi_1\not = 0$ but $\eta_1 = 0$, then a general element in $\boldsymbol{M}$ is \begin{align*}
    \eta^{2\xi} = \underbrace{\eta_1^{2\xi_1}}_{=0}\prod_{i\not=j}\eta_i^{2\xi_i}=0,
\end{align*} we have $\boldsymbol{M}=\boldsymbol{0}$ and this proves the block diagonal structure of $\boldsymbol{A}_k$.
\end{proof}

\begin{lemma}\label{lemma: equivalence off}
Let $n,p\in\N$, let $1\leq k\leq n-\kappa$ and $1\leq j\leq {n-\kappa\choose k}$, matrices $\boldsymbol{D}_m$ are generated by $\mathcal{I}^+(m,p)$, $m\in\N:m\leq n$. Then
\begin{align*}
    \boldsymbol{B}_{k,j} = \boldsymbol{D}_{n-\kappa}.
\end{align*} Moreover, $\boldsymbol{A}_0 = \boldsymbol{D}_n$ and $\boldsymbol{A}_{n-\kappa} = \boldsymbol{D}_K$.
\end{lemma}
\begin{proof}
The proof is highly similar with \cref{lemma:equivalence of B and D} and we skip the details.
\end{proof}
By well-known properties of the block upper-triangular and block diagonal matrices, the non-singularity of coefficient matrix $\boldsymbol{K}$ \cref{MatrixK} is determined by $\boldsymbol{D}_n$, generated by $\mathcal{I}^+(n,p)$, where $n,p\in\N$ is arbitrary. Since $\mathcal{I}^+(n,p)$ in $\kappa\not= 0$ case (see \cref{tab:notations off}) is identical with the $\kappa=0$ counterpart (see \cref{tab:notations on}), \cref{theorem:structure off,,lemma: equivalence off,,corollary:Dn is nonsingular} implies the following result:
\begin{corollary}
Let $n,p\in\N$, the coefficient matrix $\boldsymbol{K}$ \cref{MatrixK} for $\kappa\not=0$ is non-singular.
\end{corollary}

%%%%%%%%%%%%%%%%%%%%%%%%%%%%%%%%%%%%%%%%%%%%%%%%%%%%%%%%%%%%%%%%%%%%%%%%%%%%%%%%%%%%%%%%%%%%%%%%%%%%%%%%%%%%
\section{Numerical Results}\label{section:Num Res}
We illustrate the theoretical result \cref{theorem:MainTheorem}, namely the accuracy of the modified trapezoidal rule \cref{SecondCorrectedTrapezRule}, by presenting two numerical examples in 3D. The weakly singular kernels are chosen as\begin{align}
    s_1(x) &= \frac{x_1^2}{(x_1^2 + x_2^2 + x_3^2)^{\frac{7}{4}}}, \label{eq:s_1}\\
    s_2(x) &= \frac{x_1}{x_1^2 + x_2^2 + x_3^2}, \label{eq:s_2}
\end{align} and the regular part is \begin{align}
    \phi(x) = (1 + x_1 + x_1^2)(1 + x_2 + x_2^2)(1 + x_3 + x_3^2)\max((1 - x_1^2 - x_2^2 - x_3^2)^9,0). \label{eq:reg_part 2}
\end{align} Note that $s_1$ satisfies $\kappa = 0$ and $\delta = 1.5$, $s_2$ has $\kappa = 1$ and $\delta = 2$, the regular part $\phi\in C_c^8(\R^3)$ with $supp(\phi) =\{x\in\R^3:|x| \leq 1\}\subset\R^3$. We denote the exact values of the weakly singular integrals by \begin{align}
    J_1 &= \int_{\R^3}\phi(x)s_1(x)\d x = \frac{148281598410752}{943446919389975}\pi, \label{eq:J_1}\\
    J_2 &= \int_{\R^3}\phi(x)s_2(x)\d x = \frac{85458944}{4504759875}\pi. \label{eq:J_2}
\end{align}

\subsection{Computation of Correction Weights}
In order to use the modified trapezoidal rules \cref{SecondCorrectedTrapezRule}, one needs accurate values of weights $\bar{\omega}_\beta$'s. In this subsection, we describe a numerical method for computing correction weights.

We present the computation of the correction weights for $s_1$ in \cref{eq:s_1} on the grids $\mathcal{I}(3,0)$, $\mathcal{I}(3,1)$ and $\mathcal{I}(3,2)$ and for $s_2$ in \cref{eq:s_2} on the grids  $\mathcal{I}(3,1)$, $\mathcal{I}(3,2)$ and $\mathcal{I}(3,3)$. It is worth stressing that, since $s_1$ and $s_2$ have different values of $\kappa$, the composition of $\mathcal{I}(n,p)$ in \cref{eq:grids} is different. We solve the linear system \cref{MatrixLinearSystem3} in the proof of \cref{theorem:MainTheorem} to compute the correction weights. Compared with the linear system \cref{MatrixLinearSystem}, it has the advantage of evaluating fewer limits as $h\to 0$. To find $\mathfrak{C}(\alpha)$, which is the RHS of the system \cref{MatrixLinearSystem3}, we need $c(h)$ in \cref{RHSofLinearSystem}. We choose a radially symmetric, Schwartz function \begin{align}
    g(x) = \exp(-|x|^8), \label{eq:pract C_c}
\end{align} The reason that we choose $g$ as in \cref{eq:pract C_c} is that we can analytically evaluate the weakly singular integrals \begin{align*}
    \int_{\mathbb{R}^3\setminus\{0\}}g(x)s_1(x)x^{2\xi}\ \d x = \frac{2\Gamma(1.5 + \xi_1)\Gamma(0.5+\xi_2)\Gamma(0.5+\xi_3)\Gamma(\sfrac{(1.5+2\xi_1+2\xi_2+2\xi_3)}{8})}{8\Gamma(2.5+\xi_1+\xi_2+\xi_3)},\quad \xi\in \mathcal{I}(3,2),
\end{align*} and \begin{align*}
    \int_{\mathbb{R}^3\setminus\{0\}}g(x)s_2(x)x^{2\xi-e_1}\ \d x = \frac{2\Gamma(0.5 + \xi_1)\Gamma(0.5+\xi_2)\Gamma(0.5+\xi_3)\Gamma(\sfrac{(1+2\xi_1+2\xi_2+2\xi_3)}{8})}{8\Gamma(1.5+\xi_1+\xi_2+\xi_3)}, \quad \xi\in \mathcal{I}(3,3).
\end{align*} Combining \cref{MatrixLinearSystem3} with \cref{corollary:Taylor for C}, we have \begin{align}
    \boldsymbol{K}\bar{\omega} = \mathfrak{C}(\alpha) = c(h) + h^{2p-\kappa+2}\sum_{j\geq0}C_j h^{2j}. \label{eq:Compute weights by Richardson}
\end{align} Here, for $s_1$ case, we have $p=2$ and $\kappa = 0$; For $s_2$ case, we have $p=3$ and $\kappa = 1$. \Cref{eq:Compute weights by Richardson} enables us to apply twice the Richardson extrapolation to find $\mathfrak{C}(\alpha)$:\begin{align*}
    c^{(1)}(h) = c\left(\frac{h}{2}\right) + \frac{c\left(\frac{h}{2}\right) - c(h)}{2^{2p-\kappa+2}-1},\quad c^{(2)}(h) = c^{(1)}\left(\frac{h}{2}\right) + \frac{c^{(1)}\left(\frac{h}{2}\right) - c^{(1)}(h)}{2^{2p-\kappa+4}-1}.
\end{align*} We obtain the weights with 20 correct digits by solving $\bar{\omega} \approx \boldsymbol{K}^{-1}c^{(2)}(h)$ and ensuring $|c^{(2)}(h) - c^{(2)}(\frac{h}{2})|<10^{-21}$ for some small enough $h$. In practice, $h = \frac{1}{32}$ is sufficiently small to give more than 20 correct digits. \cref{table:correction_weights_1} and \cref{table:correction_weights_2} provide all correction weights we need.
%%%%%%%%%%%%%%%%%%%%%%%%%%%%%%%%%%%%%%%%%
\begin{table}[h]
\caption{The correction weights for $s_1$.}
\begin{center}
\begin{tabular}{ c|c|c } 
\hline\hline
 $p$ & Grid points & Correction weights \\
  \hline
   0 & $\mathcal{G}_{0,0,0} = \{(0,0,0)\}$ & $\omega_{0,0,0} = 1.6075733114131817281  $ \\
  \hline
   \multirow{4}{*}{1} & $\mathcal{G}_{0,0,0} = \{(0,0,0)\}$ & $\omega_{0,0,0} = 1.4237441285753376522   $ \\ 
   & $\mathcal{G}_{0,0,1} = \{(0,0,\pm 1)\}$ & $\omega_{0,0,1} = 0.097588595336840260411$ \\ 
   & $\mathcal{G}_{0,1,0} = \{(0,\pm 1,0)\}$ & $\omega_{0,1,0} = 0.097588595336840260411$\\ 
   & $\mathcal{G}_{1,0,0} = \{(\pm 1,0,0)\}$ & $\omega_{1,0,0} = -0.10326259925475848286$\\
  \hline
  \multirow{10}{*}{2}   & $\mathcal{G}_{0,0,0} = \{(0,0,0)\}$ & $\omega_{0,0,0} = 1.3984618420604290732 $ \\ 
  & $\mathcal{G}_{0,0,1} = \{(0,0,\pm 1)\}$ & $\omega_{0,0,1} = 0.10713725390633714656$ \\ 
   & $\mathcal{G}_{0,0,2} = \{(0,0,\pm 2)\}$ & $\omega_{0,0,2} = -0.0080179178535551260516$ \\ 
    & $\mathcal{G}_{0,1,0} = \{(0,\pm 1, 0)\}$ & $\omega_{0,1,0} = 0.10713725390633714656$ \\
   & $\mathcal{G}_{0,1,1} = \{(0,\pm 1,\pm 1),(0,\mp 1,\pm 1)\}$ & $\omega_{0,1,1} = 0.010574715875272435133$ \\
   & $\mathcal{G}_{0,2,0} = \{(0,\pm 2, 0)\}$ & $\omega_{0,2,0} = -0.0080179178535551260516$ \\
   & $\mathcal{G}_{1,0,0} = \{(\pm 1,0,0)\}$ & $\omega_{1,0,0} = -0.12143612134308144639$ \\
   & $\mathcal{G}_{1,0,1} = \{(\pm 1,0,\pm 1),(\mp 1,0,\pm 1)\}$ & $\omega_{1,0,1} = 0.00068679054708937389563$ \\
   & $\mathcal{G}_{1,1,0} = \{(\pm 1,\pm 1,0),(\mp 1,\pm 1,0)\}$ & $\omega_{1,1,0} = 0.00068679054708937389563$ \\
   & $\mathcal{G}_{2,0,0} = \{(\pm 2,0,0)\}$ & $\omega_{2,0,0} = 0.0038565899749913669879 $ \\
 \hline\hline
\end{tabular}
\end{center}
\label{table:correction_weights_1}
\end{table}

\begin{table}[h]
\caption{The correction weights for $s_2$.}
\begin{center}
\begin{tabular}{ c|c|c } 
\hline\hline
 $p$ & Grid points & Correction weights \\
  \hline
   1 & $\mathcal{G}_{1,0,0} = \{(\pm 1,0,0)\}$ & $\omega_{1,0,0} = \sfrac{1}{6}$ \\
  \hline
  \multirow{4}{*}{2}   & $\mathcal{G}_{1,0,0} = \{(\pm 1,0,0)\}$ & $\omega_{1,0,0} = 0.172099682280587019$ \\ 
   & $\mathcal{G}_{1,0,1} = \{(\pm 1, 0 , \mp 1),(\pm 1, 0 , \pm 1)\}$ & $\omega_{1,0,1} = 0.01879595247811320125$ \\ 
    & $\mathcal{G}_{1,1,0} = \{(\pm 1,\pm 1,0),(\mp 1,\pm 1,0)\}$ & $\omega_{1,1,0} = 0.01879595247811320125$\\ 
    & $\mathcal{G}_{2,0,0} = \{(\pm 2,0,0)\}$ & $\omega_{2,0,0} = -0.04030841276318657868$\\ 
  \hline
  \multirow{10}{*}{3}   & $\mathcal{G}_{1,0,0} = \{(\pm 1,0,0)\}$ & $\omega_{1,0,0} = 0.1765136604074361107$ \\ 
    & $\mathcal{G}_{1,0,1} = \{(\pm 1, 0 , \mp 1),(\pm 1, 0 , \pm 1)\}$ & $\omega_{1,0,1} = 0.02781376632443755434$ \\ 
   & $\mathcal{G}_{1,0,2} = \{(\pm 1,0,\pm 2),(\mp 1,0,\pm 2)\}$ & $\omega_{1,0,2} = -0.001785880694368878088$ \\ 
    & $\mathcal{G}_{1,1,0} = \{(\pm 1, \pm 1,0),(\mp 1, \pm 1,0)\}$ & $\omega_{1,1,0} = 0.02781376632443755434$ \\
   &$\mathcal{G}_{1,1,1} = \{(1,1,1),\cdots,(-1,-1,-1)\}$ & $\omega_{1,1,1} = 0.002634615854313335809$ \\
   & $\mathcal{G}_{1,2,0} = \{(\pm 1,\pm 2,0),(\pm 1,\mp 2,0)\}$ & $\omega_{1,2,0} = -0.001785880694368878088$ \\
   & $\mathcal{G}_{2,0,0} = \{(\pm 2,0,0)\}$ & $\omega_{2,0,0} = -0.06112550652977502187$ \\
   & $\mathcal{G}_{2,0,1} = \{(\pm 2,0,\pm 1),(\pm 2,0,\mp 1)\}$ & $\omega_{2,0,1} = -0.003571761388737756176$ \\
   & $\mathcal{G}_{2,1,0} = \{(\pm 2,\pm 1,0),(\pm 2,\mp 1,0)\}$ & $\omega_{2,1,0} = -0.003571761388737756176$ \\
   & $\mathcal{G}_{3,0,0} = \{(\pm 3,0,0)\}$ & $\omega_{3,0,0} = 0.008776034830384866974$ \\
 \hline\hline
\end{tabular}
\end{center}
\label{table:correction_weights_2}
\end{table}

\subsection{Order of Convergence of the Modified Trapezoidal Rule \texorpdfstring{$Q_h^p$}{TEXT}}
We verify the order of convergence of the corrected trapezoidal rule \cref{SecondCorrectedTrapezRule} for the weakly singular integrals $J_1$ and $J_2$ in \cref{eq:J_1,,eq:J_2}. We use $\phi$ in \cref{eq:reg_part 2} to check the order of accuracy in \cref{theorem:MainTheorem} for $s_1$ with $p=0, 1, 2$ and $s_2$ with $p = 1, 2, 3$. We evaluate $|I - Q_h^p|$ for mesh-sizes $h = \frac{1}{2^3},\cdots, \frac{1}{2^{7}}$, where $I$ is the true values of integrals $J_1$ and $J_2$, and perform linear regressions in log-log plots to find the order of convergence, as shown in \cref{fig:order}. We find that the numerical results match very well with theoretically predicted order of accuracy $2p+3.5$ with $p = 0, 1, 2$ for $s_1$ and $2p+3$ with $p = 1, 2, 3$ for $s_2$. When the value of $p$ becomes large and $h$ becomes small, the round-off errors dominate. In this case, multi-precision computation will help when one chooses large $p$.

%%%%%%%%%%%%%%%%%%%%%%%%%%%%%%%%%%%%%%%%%%%%%%%%%%%%%%%%%%%%%%%%%%%%%%%%%%%%%%%%%%%%

    \begin{figure}[h!]
    \begin{subfigure}[b]{0.5\textwidth}
         \centering
         \includegraphics[width=\textwidth]{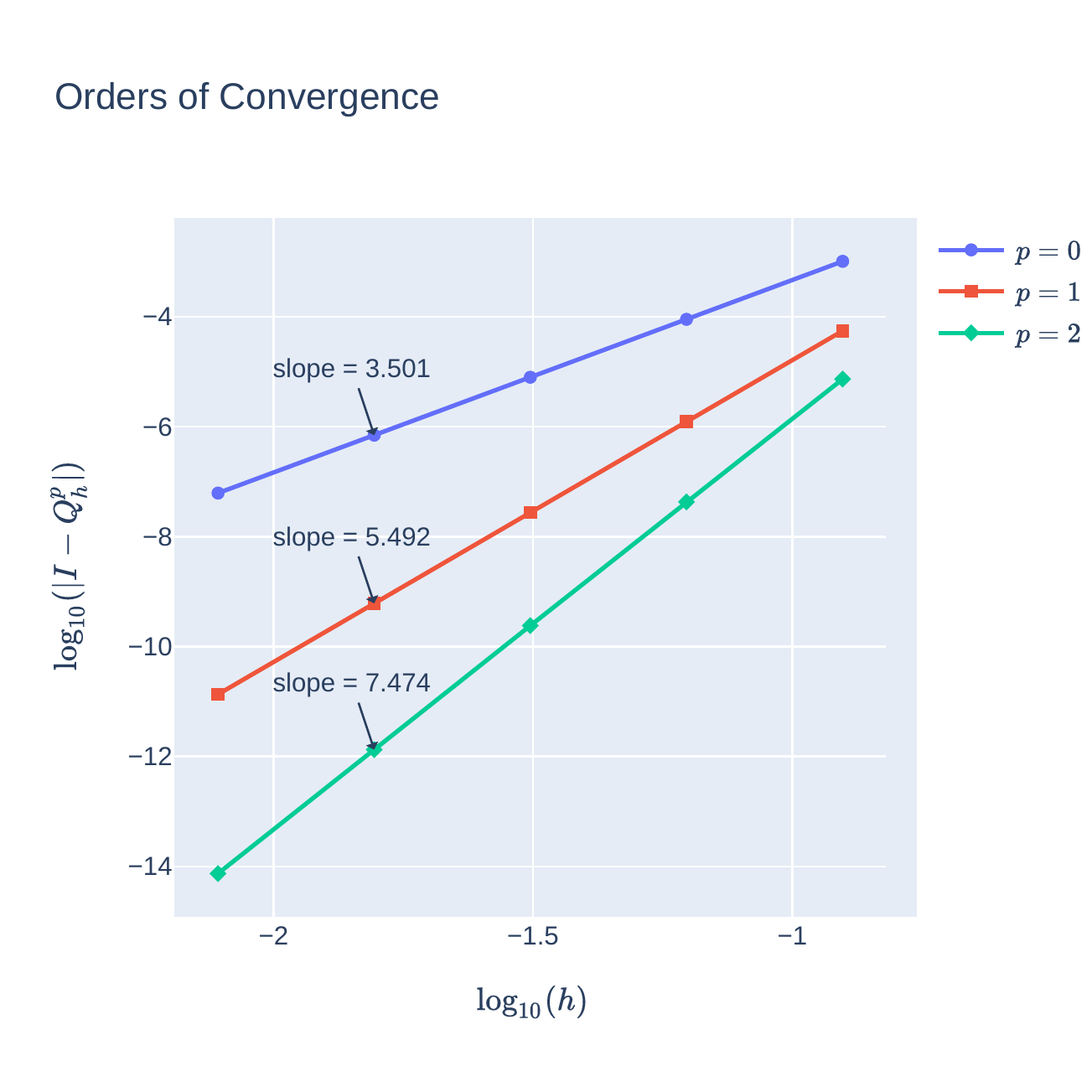}
         \caption{$s_1$ }
         \label{fig:On-diag}
     \end{subfigure}
     \hfill
         \begin{subfigure}[b]{0.5\textwidth}
         \centering
         \includegraphics[width=\textwidth]{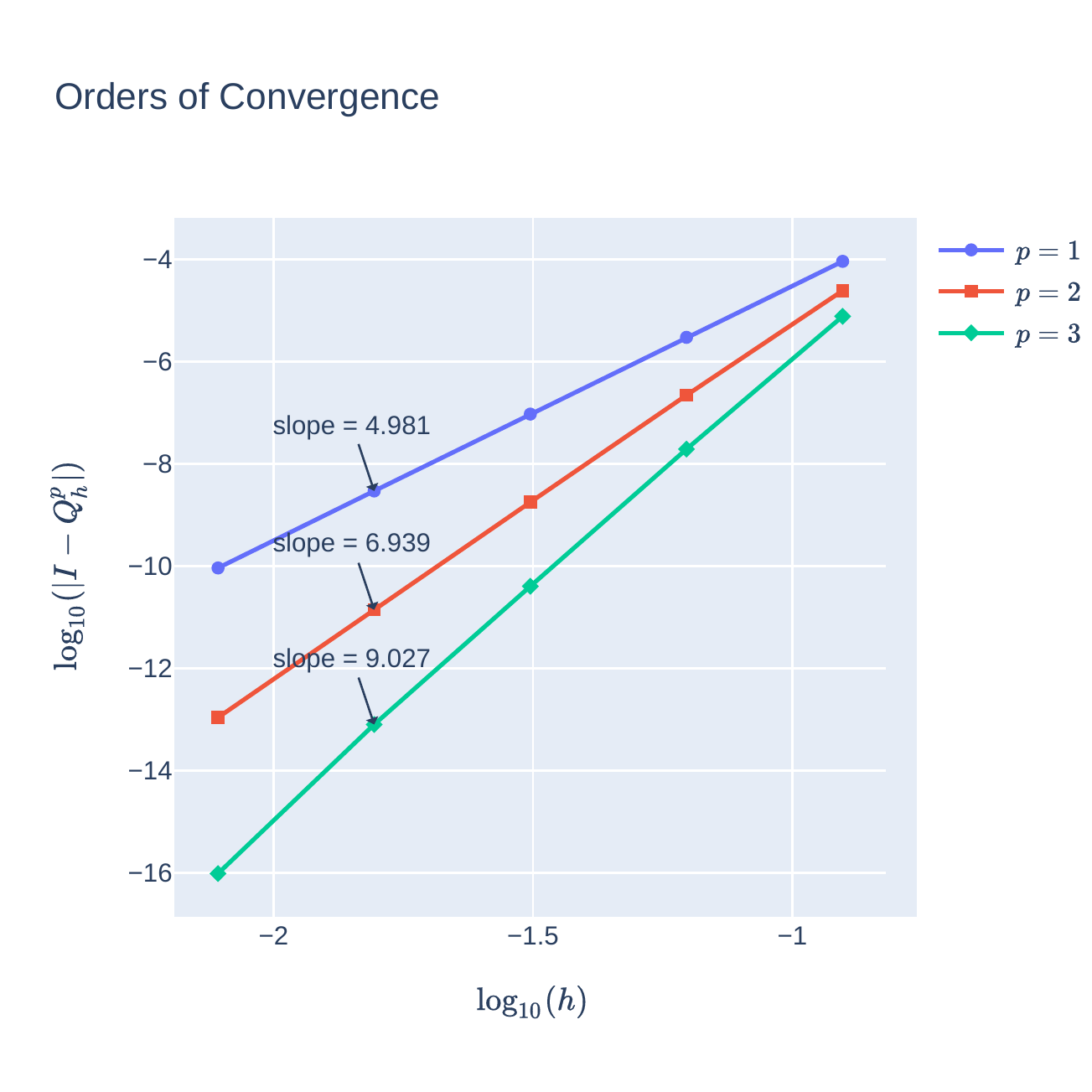}
         \caption{$s_2$ }
         \label{fig:Riesz}
     \end{subfigure}
    \caption{Numerical orders of accuracy: Log-log plots of the quadrature error against $h$ for the modified trapezoidal rule corresponding to the weakly singular kernels (a) $s_1$ and (b) $s_2$.}
    \label{fig:order}
\end{figure}

%%%%%%%%%%%%%%%%%%%%%%%%%%%%%%%%%%%%%%%%%%%%%%%%%%%%%%%%%%%%%%%%%%%%%%%%%%%%%%%%%%%%%%%%%%%%%%%%%%%%%%%%%

\section{Conclusion}\label{sec:conclusion}
We propose an arbitrarily high-order modified trapezoidal rules for a large class of weakly singular integrals with singular part satisfy the dilation property \cref{eq:dilation cond} and symmetry property \cref{eq:symmetry}, and some sufficient smooth function $\phi$ with compact support. The rule is a punctured-hole trapezoidal rule with correction terms. We have shown the order of accuracy of the modified quadrature given the number of correction layers. We focus on the error due to singularity in this paper. The rule can be combined with any boundary error correction for regular functions $\phi$ without compact support to attain high-order convergence. The correction weights can be pre-computed and stored for future use. We tabulate the correction weights required for the two numerical examples with 20 correct digits. We provide theoretical guarantee that the rule works with arbitrary ``correction-layers'' $p$ in arbitrary $n$ dimensions, though additional treatment is necessary whenever large $p$ is used (e.g. $p > 10$), since numerical evidence suggests that linear systems \cref{MatrixLinearSystem} become increasingly ill-conditioned for large $p$. For future works, one can further relax the smoothness criteria for $\phi$ in \cref{theorem:MainTheorem}. In addition, we can adapt the modified trapezoidal rule to other common weakly singular kernels that does not satisfy our hypotheses \cref{eq:dilation cond,,eq:symmetry}, such as $x_1\log{(|x|)}$.

\bibliographystyle{plain}
\bibliography{main}
\end{document}